\documentclass[12pt]{article}
\input{style}
\usepackage[color = orange!10]{todonotes}
\usepackage{enumitem}
\usepackage{varwidth}
\newcounter{item}

\begin{document}
\title{Backward martingale transport and Fitzpatrick functions in
  pseudo-Euclidean spaces}

\author{Dmitry Kramkov\footnote{Carnegie Mellon University, Department
    of Mathematical Sciences, 5000 Forbes Avenue, Pittsburgh, PA,
    15213-3890, USA,
    \href{mailto:kramkov@cmu.edu}{kramkov@cmu.edu}. The author also
    has a research position at the University of Oxford.}  and Mihai
  S\^{i}rbu\footnote{The University of Texas at Austin, Department of
    Mathematics, 2515 Speedway Stop C1200, Austin, Texas 78712,
    \href{mailto:sirbu@math.utexas.edu}{sirbu@math.utexas.edu}. The
    research of this author was supported in part by the National
    Science Foundation under Grant DMS 1908903.}  }

\date{\today}

\maketitle

\begin{abstract}
  We study an optimal transport problem with a backward martingale
  constraint in a pseudo-Euclidean space $S$.  We show that the dual
  problem consists in the minimization of the expected values of the
  Fitzpatrick functions associated with maximal $S$-monotone sets. An
  optimal plan $\gamma$ and an optimal maximal $S$-monotone set $G$
  are characterized by the condition that the support of $\gamma$ is
  contained in the graph of the $S$-projection on $G$. For a Gaussian
  random variable $Y$, we get a unique decomposition: $Y = X+Z$, where
  $X$ and $Z$ are independent Gaussian random variables taking values,
  respectively, in complementary positive and negative linear
  subspaces of the $S$-space.
\end{abstract}

\begin{description}
\item[Keywords:] martingale optimal transport, pseudo-Euclidean space,
  Fitzpatrick function.
\item[AMS Subject Classification (2010):] 60G42, 91B24, 91B52.
  % \item[JEL Classification:]
\end{description}

% \listoftodos

\section{Introduction}

Let $S$ be a symmetric $d\times d$ matrix with full rank and $m$
positive eigenvalues. The bilinear form
\begin{displaymath}
  S(x,y) \set \ip{x}{Sy}  = \sum_{i,j=1}^d x^i S_{ij}y^j, \quad
  x,y\in \real{d},
\end{displaymath}
defines the scalar product on a pseudo-Euclidean space $\real{d}_m$
with dimension $d$ and index $m$, which we call the $S$-space. The
quadratic form $S(x,x)$ is called the \emph{scalar square}; its value
may be negative.

Given a Borel probability measure $\nu$ on $\real{d}$ with finite
second moments, we study the backward martingale transport problem:
\begin{equation}
  \label{eq:1}
  \text{maximize} \quad  \frac 12 \int S(x, y)
  d\gamma\quad\text{over}\quad 
  \gamma\in \Gamma(\nu), 
\end{equation}
where $\Gamma(\nu)$ is the family of Borel probability measures
$\gamma = \gamma(dx,dy)$ on $\real{2d}$ having $\nu$ as their
$y$-marginal and the martingale property: $\gamma(y|x) = x$. Contrary
to the usual setup of optimal transport, with or without the
martingale constraints, the $x$-marginal of $\gamma$ is not an input,
but a part of the solution. In the case $m=1$ and $S$ given
by~\eqref{eq:2}, problem~\eqref{eq:1} has been studied
in~\cite{KramXu:22}, where it was motivated by an application from
financial economics.

We point out that the problem of unconstrained optimal transport with
\emph{given} $x$-marginal $\mu$ and the quadratic cost function
$c(x,y) = -S(x,y)$ can be reduced to the classical
$\lsp{2}{d}$-framework of~\cite{Bren:91} by a linear
transformation. The same problem with a martingale constraint is
trivial, as
\begin{displaymath}
  \int S(x, y)
  d\gamma = \int S(x,x) d\mu. 
\end{displaymath}
For \emph{nonquadratic} cost functions, such martingale optimal
transport has been the subject of intensive research, see, for
example, \cite{BeigJuil:16, BeigCoxHues:17, BeigNutzTouz:17,
  HenrTouz:16}, and \cite{GhousKimLim:19}.

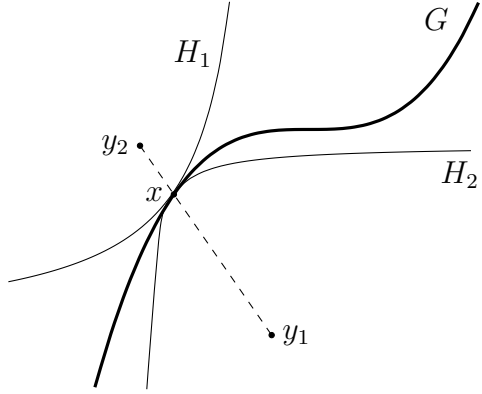
\begin{figure}
  \centering
  \begin{tikzpicture}[scale = 0.5]
    % \draw[help lines] (-8,-6) grid (8,6);
   
    \draw[dashed, thin] (-4.5,0.568) node [left] {$y_2$} --
    (-3.6,-0.728) node [left] {$x$} --(-1, -4.472) node [right]
    {$y_1$};

    \filldraw [black] (-4.5,0.568) circle [radius=2pt] (-3.6,-0.728)
    circle [radius=2pt] (-1,-4.472) circle [radius=2pt];

    \draw[domain=-4.32:4.3,smooth,variable=\x,black] plot
    ({\x},{-1.1664/(\x+4.5) +0.568}); \draw (4,0.4452) node [below]
    {$H_2$};

    \draw[domain=-8:-2.12,smooth,variable=\x,black] plot
    ({\x},{-9.9344/(\x+1) - 4.472}); \draw (-2.3,3.016) node [left]
    {$H_1$};

    \draw[very thick,domain=-5.7:4.5,smooth,variable=\x,black] plot
    ({\x},{(\x/3)^3 + 1}); \draw (4, 3.37) node [above left] {$G$};

  \end{tikzpicture}
  \caption{The convex set $Q_G(x)\subset P_G^{-1}(x)$ contains the
    points $y_i$, $i=1,2$. Hyperbola $H_i$ with focus at $y_i$ is
    tangent to $G$ at $x$.}
  \label{fig:1}
\end{figure}

Our main Theorem~\ref{th:1} states that $\gamma$ is an optimal plan
for~\eqref{eq:1} if and only if its support is contained in the
$y$-based graph of the projection multi-function
\begin{displaymath}
  P_G(y)\set \argmin_{x\in G} S(x-y,x-y), \quad y\in \real{d}, 
\end{displaymath}
on a maximal $S$-monotone set $G$. Here, a set $G\subset \real{d}$ is
called $S$-\emph{monotone} or $S$-\emph{positive} if
\begin{displaymath}
  S(x-y,x-y) \geq 0, \quad x,y\in G,   
\end{displaymath}
and \emph{maximal} $S$-monotone, if it is not a strict subset of an
$S$-monotone set. If $d=2m$ and the matrix $S$ has the form:
\begin{equation}
  \label{eq:2}
  S = 
  \begin{pmatrix}
    0 & I \\
    I & 0
  \end{pmatrix},
\end{equation}
where $I$ is the $m$-dimensional identity matrix, then the
$S$-monotonicity becomes the classical monotonicity: see, for example,
\cite[p.~240]{Rock:70}. Geometrically, $P_G(y)$ consists of those
$x\in G$ where the hyperboloid
\begin{displaymath}
  H_G(y) \set \descr{z\in \real{d}}{S(y-z,y-z)= \min_{w\in G} S(y-w,y-w)}
\end{displaymath}
is \emph{tangent} to $G$, as in Figure \ref{fig:1}. If $x\in P_G(y)$,
then the vector $y-x$ is $S$-\emph{normal} to $G$ at $x$ in the sense
that
\begin{displaymath}
  \limsup_{z\to x, z\in G} \frac{S(y-x,z-x)}{\abs{z-x}} \leq 0. 
\end{displaymath}

The map $y\to P_G(y) \ni x$ taking values in the closed subsets of $G$
describes the \emph{backward} evolution of the canonical martingale
$(x,y)$ under $\gamma$.  In terms of the \emph{forward evolution}, an
optimal plan is characterized as
\begin{displaymath}
  x\in G, \quad y\in Q_G(x), \quad \as{\gamma}, 
\end{displaymath}
where $Q_G(x)$ is the smallest closed convex subset of
\begin{displaymath}
  P^{-1}_G(x)\set \descr{y\in \real{d}}{P_G(y)=x}
\end{displaymath}
containing $x$ in its relative interior. If $y\in Q_G(x)$, then the
vector $y-x$ is $S$-\emph{orthogonal} to $G$ at $x$ in the sense that
\begin{displaymath}
  \lim_{z\to x, z\in G} \frac{S(y-x,z-x)}{\abs{z-x}} = 0. 
\end{displaymath}

The associated dual problem to~\eqref{eq:1} turns out to be
\begin{displaymath}
  \text{minimize} \quad \int {\psi_G(y)} d\nu \quad\text{over}\quad G \in
  \mset{S}, 
\end{displaymath}
where $\mset{S}$ is the family of all maximal $S$-monotone sets and
\begin{displaymath}
  \psi_G(y) \set \sup_{x\in G} \cbraces{S(x,y) - \frac12 S(x,x)},
  \quad y\in \real{d}. 
\end{displaymath}
For $S$ given by~\eqref{eq:2}, $\psi_G$ is the classical Fitzpatrick
function from convex analysis, see~\cite{Fitzpatrick:88}
and~\cite{PenotZalin:05}.  We shall use this term also for general
$S$-spaces. The Fitzpatrick functions $\psi_G$ play in our analysis
the role of the Kantorovich potentials from unconstrained optimal
transport.  Monotone sets in $S$-spaces and their corresponding
Fitzpatrick functions have been introduced in~\cite{Simons:07} and
further studied in~\cite{Penot:09}.

Problems~\eqref{eq:1} and~\eqref{eq:2} can be solved explicitly if
$\nu$ is the law of a Gaussian random variable $Y$ with positive
definite covariance matrix $\Sigma$.  Theorem~\ref{th:10} shows that
the optimal plan $\gamma$ for~\eqref{eq:1} is given by the law of
$(X,Y)$, where $X$ is the unique random variable such that $X$ and
$Z\set Y-X$ are independent Gaussian and their respective supports are
complementary positive and negative linear subspaces of the $S$-space:
\begin{displaymath}
  S(X,X) \geq 0, \quad S(Z,Z) \leq 0, \quad S(X,Z) = 0.
\end{displaymath}
The covariance matrices $Q$ for $X$ and $R$ for $Z$ are uniquely
defined by the conditions:
\begin{displaymath}
  \Sigma =Q+R, \quad QSQ \geq 0, \quad RSR \leq 0, \quad QSR =
  0,   
\end{displaymath}
where inequalities mean that the respective matrices are positive and
negative semi-definite.

The original motivation for the backward martingale transport comes
from the Kyle's equilibrium for insider trading introduced
in~\cite{Kyle:85}.  For $m=1$ and $S$ given by \eqref{eq:2}, the
paper~\cite{KramXu:22} provides sharp conditions for the existence and
uniqueness of a version of such equilibrium from~\cite{RochVila:94}.
Our main Theorem~\ref{th:1} sets up the base for an extension of these
criteria to higher dimensions.

Our results also provide pseudo-Euclidean counterparts to the
solutions of the classical problems of linear and nonlinear PCA
(principal component analysis).  The optimal linear hyperplane
$G = \range{Q}$ given by Theorems \ref{th:10} and \ref{th:6} is the
PCA fit of the multivariate Gaussian distribution in the $S$-space:
$G$ is the span of the eigenvectors corresponding to the largest
(equivalently, positive) $m$ eigenvalues of the characteristic
equation $\det(S-\lambda \Sigma ^{-1})=0$.  The optimal maximal
$S$-monotone set $G$ from Theorem \ref{th:1} is a pseudo-Euclidean
version of a self-consistent or principle curve
from~\cite{HasStu:89}. Notice that being maximal $S$-monotone, the set
$G$ is a maximal \emph{Euclidean-type} subset of $\real{d}$ in the
$S$-space: $S(x-y,x-y)\geq 0$, $x,y\in G$.  As a result, $G$ has
qualitative properties of a standard Euclidean PCA fit.

For both applications above, it is important to know whether an
optimal plan $\gamma$ admits the \emph{map representation}: there is a
Borel function $f$ on $\real{d}$ such that
\begin{displaymath}
  x = f(y), \quad \as{\gamma(dx,dy)} 
\end{displaymath}
The follow-up paper \cite{KramSirb:23} contains sharp conditions for
the existence of such representation as well as for the uniqueness of
an optimal plan $\gamma$.

The current paper is organized as follows. The existence and
characterizations of an optimal plan $\gamma$ and an optimal set $G$
are established in Theorem~\ref{th:1} in
Section~\ref{sec:3}. Properties of optimal $\gamma$ and $G$ are
collected in Section~\ref{sec:prop-an-optim}. The linear case is
studied in Section~\ref{sec:linear}. Finally,
Appendix~\ref{sec:basic-prop-fitzp} lists the main features of
$S$-Fitzpatrick functions.

\subsection*{Notation}
\label{sec:notations}

The scalar product and the norm in the Euclidean space $\real{d}$ are
written as
\begin{displaymath}
  \ip{x}{y} \set \sum_{i=1}^d x_i y_i, \quad \abs{x} \set
  \sqrt{\ip{x}{x}}, \quad x,y\in \real{d}.
\end{displaymath}

For a Borel probability measure $\mu$ on $\real{d}$, a
$\mu$-integrable $m$-dimensional Borel function $f=(f_1,\dots,f_m)$,
and an $n$-dimensional Borel function $g = (g_1,\dots,g_n)$, the
notation $\mu(f|g)$ stands for the $m$-dimensional vector of
conditional expectations of $f_i$ given $g$ under $\mu$:
\begin{displaymath}
  \mu(f|g) = (\mu(f_1|g_1,\dots,g_n),\dots,\mu(f_m|g_1,\dots,g_n)).
\end{displaymath}
In particular, we write $\mu(f)$ for the vector of expected values:
\begin{displaymath}
  \mu(f) = \int f d\mu = (\int f_1 d\mu, \dots, \int f_m d\mu) =
  (\mu(f_1),\dots,\mu(f_m)). 
\end{displaymath}
We write $\supp{\mu}$ for the support of $\mu$, the smallest closed
set of full measure.  For $p\geq 1$, we denote by $\ps{p}{\real{d}}$
the family of Borel probability measures $\mu$ on $\real{d}$ such that
$\mu(\abs{x}^p) = \int \abs{x}^p d\mu < \infty$.

For a closed (equivalently, lower semi-continuous) convex function
$\map{f}{\real{d}}{\realext}$, we denote by $\partial f$ its
subdifferential and by $f^*$ its convex conjugate:
\begin{align*}
  \partial f(x) &= \descr{z\in \real{d}}{f(y)-f(x) \geq \ip{z}{y-x},
                  \; y\in \real{d}}, \quad x\in \real{d}, \\
  f^*(y) &\set \sup_{x\in \real{d}}\cbraces{\ip{y}{x} - f(x)}, \quad
           y\in \real{d}.
\end{align*}
We recall that, \cite{Rock:70}, Theorem~23.5, p.~218,
\begin{equation}
  \label{eq:3}
  f(x)=\ip{x}{y}-f^*(y)\iff y\in 
  \partial f(x) \iff x\in \partial f^*(y), \quad x,y\in \real{d}. 
\end{equation} 

The domains of $f$ and of $\partial f$ are defined as
\begin{align*}
  \dom f &\set \descr{x\in \real{d}}{f(x)<\infty}, \\
  \dom \partial f & \set \descr{x\in \dom f}{\partial f(x) \not= \emptyset}. 
\end{align*}

\section{Backward martingale transport}
\label{sec:3}

We denote by $\msym{d}{m}$ the family of symmetric
$d\times d$-matrices of full rank with $m \in \braces{0,1,\dots,d}$
positive eigenvalues. For $S \in \msym{d}{m}$, the bilinear form
\begin{displaymath}
  S(x,y) \set \ip{x}{Sy}  = \sum_{i,j=1}^d x^i S_{ij}y^j, \quad
  x,y\in \real{d},
\end{displaymath}
defines the scalar product on a pseudo-Euclidean space $\real{d}_m$
with dimension $d$ and index $m$, which we call the $S$-space. The
quadratic form $S(x,x)$ is called the \emph{scalar square}; its value
may be negative.

We view a point $(x,y)$ in the product space
$\real{2d} = \real{d}\times \real{d}$ as the initial and terminal
values of the canonical $d$-dimensional process. Given the terminal
law $\nu = \nu(dy) \in \ps{2}{\real{d}}$, our goal is to
\begin{equation}
  \label{eq:4}
  \text{maximize} \quad \frac12 \int S(x, y)
  d\gamma\quad\text{over}\quad 
  \gamma\in \Gamma(\nu), 
\end{equation}
where $\Gamma(\nu)$ is the family of Borel probability measures
$\gamma = \gamma(dx,dy) \in \ps{2}{\real{2d}}$ having $\nu$ as their
$y$-marginal and making a martingale out of the canonical process:
\begin{displaymath}
  \Gamma(\nu) \set \descr{\gamma \in
    \ps{2}{\real{2d}}}{\gamma(\real{d},dy) =  
    \nu(dy) \text{ and } \gamma(y|x)=x}.  
\end{displaymath}
The martingale property of $\gamma \in \Gamma(\nu)$ yields that
\begin{align*}
  \int S(x,y) d\gamma & = \int S(x,x) d\gamma, \\
  \int S(x-y, x-y) d\gamma &= \int S(y,y) d\nu - \int S(x,y)
                             d\gamma.
\end{align*}
Thus, the problem~\eqref{eq:4} is equivalent to the minimization of
the expected value of the scalar square between $x$ and $y$ over
$\Gamma(\nu)$:
\begin{equation}
  \label{eq:5}
  \text{minimize} \quad  \int S(x-y, x-y)
  d\gamma\quad\text{over}\quad 
  \gamma\in \Gamma(\nu). 
\end{equation}
Theorem~\ref{th:3} contains another reformulation of~\eqref{eq:4}.

A set $G\subset \real{d}$ is called $S$-\emph{monotone} or
$S$-\emph{positive} if
\begin{displaymath}
  S(x-y,x-y)  \geq 0, \quad x, y
  \in G,   
\end{displaymath}
that is, if the restriction to $G$ of the scalar square takes values
in $[0,\infty)$.  An $S$-monotone set $G$ is called \emph{maximal} if
it is not a strict subset of an $S$-monotone set.  Theorem~\ref{th:1}
shows that a dual problem to~\eqref{eq:4} is to
\begin{equation}
  \label{eq:6}
  \text{minimize} \int \psi_G(y) d\nu \quad\text{over}\quad
  G \in \mset{S}, 
\end{equation}
where $\mset{S}$ is the family of all maximal $S$-monotone sets and
\begin{displaymath}
  \psi_G(y) \set  \sup_{x\in G} \cbraces{S(x,y) - \frac12 S(x,x)}, \quad 
  y\in \real{d}, 
\end{displaymath}
is a closed convex function taking values in $\realext$. For general
quadratic forms $S$ such function $\psi_G$ was introduced
in~\cite{Simons:07}. We call $\psi_G$ a \emph{Fitzpatrick} function
because of Example~\ref{ex:3}.  We collect the basic properties of
Fitzpatrick functions in the $S$-space in
Appendix~\ref{sec:basic-prop-fitzp}.

\begin{Example}
  \label{ex:1}
  If $S \in \msym{d}{d}$, that is, $S$ is positive definite:
  \begin{displaymath}
    S(x,x) = \ip{x}{Sx} > 0, \quad x\not=0, 
  \end{displaymath}
  then $G=\real{d}$ is the only element of $\mset{S}$ and
  \begin{displaymath}
    \psi_{G}(y) = \sup_{x\in \real{d}} \cbraces{S(x,y) - \frac12
      S(x,x)} = \frac12 S(y,y), \quad y\in \real{d}. 
  \end{displaymath}
  Trivially, $G=\real{d}$ is the optimal set
  for~\eqref{eq:6}. From~\eqref{eq:5} we deduce that $x=y$,
  $\as{\gamma}$, under the optimal plan $\gamma$ for~\eqref{eq:4}.
\end{Example}

\begin{Example}
  \label{ex:2}
  If $S \in \msym{d}0$, that is, $S$ is negative definite:
  \begin{displaymath}
    S(x,x) = \ip{x}{Sx} < 0, \quad x\not=0, 
  \end{displaymath}
  then all maximal $S$-monotone sets are points:
  $\mset{S} = \real{d}$. If $G=\braces{x}$, where $x \in \real{d}$,
  then
  \begin{displaymath}
    \psi_G(y) =  S(x,y) - \frac12
    S(x,x), \quad y\in \real{d}.
  \end{displaymath}
  Elementary computations show that the optimal set $G$
  for~\eqref{eq:6} is given by $\braces{\nu(y)}$ and that $x=\nu(y)$,
  $\as{\gamma}$, under the optimal plan $\gamma$ for~\eqref{eq:4}.
\end{Example}

\begin{Example}
  \label{ex:3}
  If $d=2m$ and $S$ has the form~\eqref{eq:2}, that is,
  \begin{displaymath}
    S(x,y) = \sum_{i=1}^m \cbraces{x^i y^{m+i} + x^{m+i}y^i},  \quad
    x,y\in \real{2m}, 
  \end{displaymath}
  then $S \in \msym{2m}m$, the $S$-monotonicity means the standard
  monotonicity in $\real{2m}=\real{m}\times \real{m}$, and $\psi_G$
  becomes the classical Fitzpatrick function
  from~\cite{Fitzpatrick:88}. For $d=2$, problems~\eqref{eq:4}
  and~\eqref{eq:6} have been studied in~\cite{KramXu:22}, where they
  were motivated by applications in financial economics.
\end{Example}

\begin{Example}
  \label{ex:4}
  If $S$ is the \emph{canonical} quadratic form in $\msym{d}m$:
  \begin{displaymath}
    S(x,y) = \sum_{i=1}^m x^i y^i  - \sum_{i=m+1}^d x^i y^i, \quad
    x,y\in \real{d},  
  \end{displaymath}
  then a set $G$ is $S$-monotone if and only if
  \begin{displaymath}
    G = \graph{f} \set \descr{(u,f(u))}{u\in D}, 
  \end{displaymath}
  where $D\subset \real{m}$ and $\map{f}{D}{\real{d-m}}$ is a
  $1$-Lipschitz function:
  \begin{displaymath}
    \abs{f(u)-f(v)} \leq \abs{u-v}, \quad u,v \in D.  
  \end{displaymath}
  By the Kirszbraun Theorem, \cite[2.10.43]{Feder:69}, every
  $1$-Lipschitz function can be extended to the whole $\real{m}$. It
  follows that $G \in \mset{S}$ if and only if it is the graph of a
  global $1$-Lipschitz function $\map{f}{\real{m}}{\real{d-m}}$.
\end{Example}

Let $G \in \mset{S}$. The Fitzpatrick function $\psi_G$ is closely
related to the minimization of the scalar square relative to $G$:
\begin{displaymath}
  \phi_G(y) \set \inf_{x\in G}
  S(x-y,x-y)= S(y,y) - 2\psi_G(y), \quad
  y\in \real{d}. 
\end{displaymath}
Theorem~\ref{th:14} yields that $x\in G$ and
$y \in \partial \psi^*_G(Sx)$ (equivalently,
$Sx \in \partial \psi_G(y)$) if and only if $x$ is a \emph{projection
  of $y$ on $G$} in the $S$-space:
\begin{equation}
  \label{eq:7}
  \begin{split}
    x \in P_G(y) \set& \argmin_{z\in G} S(y-z,y-z) \\
    =& \descr{z\in G}{S(y-z,y-z) = \phi_G(y)} \\
    =& \descr{z\in G}{\psi_G(y) = S(y,z) - \frac12 S(z,z)}.
  \end{split}
\end{equation}
The inverse image of $P_G$ has the form:
\begin{displaymath}
  P_G^{-1}(x) \set \descr{y\in \real{d}}{x\in P_G(y)} = \partial
  \psi_G^*(Sx), \quad x\in G, 
\end{displaymath}
and thus, takes values in the convex closed sets.  Geometrically,
$x\in P_G(y)$ if and only if the hyperboloid
\begin{displaymath}
  H_G(y) \set \descr{z \in \real{d}}{S(y-z,y-z) = \phi_G(y)}
\end{displaymath}
is \emph{tangent} to $G$ at $x$.  Figure~\ref{fig:1} provides an
illustration for $d=2$ and $S=S(x,y)$ from Example~\ref{ex:3}.

Theorems~\ref{th:1} and~\ref{th:2} show that the geometric properties
of an optimal plan for~\eqref{eq:4} are fully described by the
multi-function
\begin{displaymath}
  Q_G(x) \set \bigcup_{n\geq 1} Q_G^{1/n}(x), \quad x \in G,
\end{displaymath}
where for $\epsilon >0$
\begin{align*}
  Q_G^{\epsilon}(x)
  \set  \descr{y \in P_G^{-1}(x)}{x + \epsilon(x-y) \in P_G^{-1}(x)},
  \quad x\in G.  
\end{align*}
In other words, $x\in G$ and $y \in Q_G(x)$ if and only if there are
$z \in \real{d}$ and $t\in (0,1)$ such that
\begin{displaymath}
  x = t y + (1-t) z \quad\text{and}\quad x \in P_G(y)
  \cap P_G(z). 
\end{displaymath}
Lemma~\ref{lem:6} states that $Q_G(x)$ is the largest closed convex
subset of $P^{-1}_G(x)$ whose relative interior contains $x$.

Let $0<\delta < \epsilon$. As $P^{-1}_G(x) = \partial \psi_G^*(Sx)$,
the convexity and continuity properties of subdifferentials of convex
functions yield that
\begin{gather*}
  Q_G^{\epsilon}(x) \subset Q_G^{\delta}(x), \quad
  x\in G, \\
  \graph{Q_G^{\epsilon}} \set \descr{(x,y)}{x \in G,\, y \in
    Q_G^{\epsilon}(x)} \text{ is a closed set.}
\end{gather*}
It follows that
\begin{displaymath}
  \graph{Q_G} \set \descr{(x,y)}{x \in G,\, y\in Q_G(x)}
  = \bigcup_{n\geq 1} \graph{Q_G^{1/n}}  \text{ is a
    $F_\sigma$-set}, 
\end{displaymath}
that is, a countable union of closed sets. In particular,
$\graph{Q_G}$ is a Borel set in $\real{2d}$. This fact is used
implicitly in item~\ref{item:3} of Theorem~\ref{th:1}.

The following theorem is the main result of the paper. We recall that
$\supp{\gamma}$ stands for the support, the smallest closed set of
full measure, of a Borel probability measure $\gamma$.

\begin{Theorem}
  \label{th:1}
  Let $S \in \msym{d}{m}$ and $\nu \in
  \ps{2}{\real{d}}$. Problems~\eqref{eq:4} and~\eqref{eq:6} have
  solutions and the identical values:
  \begin{equation}
    \label{eq:8}
    \max_{\gamma \in \Gamma(\nu)} \frac12 \int S(x,y)
    d\gamma = \min_{G\in 
      \mset{S}} \int \psi_G d\nu.  
  \end{equation}
  For any $\gamma \in \Gamma(\nu)$ and $G\in \mset{S}$, the following
  conditions are equivalent:
  \begin{enumerate}[label = {\rm (\alph{*})}, ref={\rm (\alph{*})}]
  \item \label{item:1} $\gamma$ is an optimal plan for~\eqref{eq:4}
    and $G$ is an optimal set for~\eqref{eq:6}.
  \item \label{item:2} {\rm (Backward evolution)} $x \in P_G(y)$ for
    every $(x,y)\in \supp{\gamma}$.
  \item \label{item:3} {\rm (Forward evolution)} $x \in G$ and
    $y \in Q_G(x)$, $\as{\gamma}$ \setcounter{item}{\value{enumi}}
  \end{enumerate}
\end{Theorem}

\begin{Remark}
  \label{rem:1}
  Theorem~\ref{th:14} and the construction of the projection
  multi-function $P_G$ in~\eqref{eq:7} yield the equivalence of
  item~\ref{item:2} to any of the following conditions:
  \begin{enumerate}[label = {\rm (\alph{*})}, ref={\rm (\alph{*})}]
    \setcounter{enumi}{\value{item}}
  \item \label{item:4} $\psi_G(y) = S(x,y) - \frac12 S(x,x)$ and
    $x\in G$ for every $(x,y)\in \supp{\gamma}$.
  \item \label{item:5} $Sx \in \partial \psi_G(y)$ and $x \in G$ for
    every $(x,y)\in \supp{\gamma}$.
  \end{enumerate}
\end{Remark}

The proof of Theorem~\ref{th:1} is divided into lemmas.

\begin{Lemma}
  \label{lem:1}
  Let $S\in \msym{d}{m}$, $\nu \in \ps{2}{\real{d}}$,
  $\gamma\in \Gamma(\nu)$, and $G\in \mset{S}$. Then
  \begin{gather}
    \label{eq:9}
    \frac12 S(x,x) \leq \psi_G(x) \leq
    \gamma(\psi_G(y)|x), \quad \as{\gamma}, \\
    \label{eq:10}
    \frac12 \int S(x,y) d\gamma = \frac12 \int S(x,x) d\gamma \leq
    \int \psi_G d\nu.
  \end{gather}
\end{Lemma}
\begin{proof}
  By Theorem~\ref{th:12}, the Fitzpatrick function $\psi_G$ is convex
  and $\psi_G(x)\geq \frac12 S(x,x)$. The second inequality
  in~\eqref{eq:9} follows now from the conditional Jensen inequality
  and the martingale property of $\gamma$. The first part
  of~\eqref{eq:10} is a consequence of the martingale property of
  $\gamma$. The second part of~\eqref{eq:10} is obtained by
  integrating~\eqref{eq:9} over $\gamma$.
\end{proof}

We equip $\mathcal{P}_2\cbraces{\real{d}}$ with the Wasserstein
2-metric:
\begin{displaymath}
  W_2(\mu,\nu)  \set \inf_{\pi \in \Pi(\mu,\nu)} \sqrt{\int
    \abs{x-y}^2 d\pi}, \quad \mu,\nu \in
  \mathcal{P}_2\cbraces{\real{d}},  
\end{displaymath}
where $\Pi(\mu, \nu)$ is the family of Borel probability measures on
$\real{2d}$ with $x$-marginal $\mu$ and $y$-marginal $\nu$.  We recall
that $(\mathcal{P}_2\cbraces{\real{d}},W_2)$ is a Polish space and
that $W_2(\mu_n,\mu)\to 0$ if and only if $\mu_n(f) \to \mu(f)$ for
every continuous function $f=f(x)$ on $\real{d}$ with quadratic
growth:
\begin{displaymath}
  \abs{f(x)} \leq K (1 + \abs{x}^2), \quad x\in \real{d}, 
\end{displaymath}
where $K=K(f)>0$ is a constant.  A set
$A\subset \mathcal{P}_2(\real{d})$ is \emph{pre-compact} under $W_2$
if and only if
\begin{equation}
  \label{eq:11}
  \sup_{\mu \in A} \int_{\abs{x}\geq K} \abs{x}^2 d\mu \to 0,
  \quad K\to \infty. 
\end{equation}
These results can be found in \cite[Chapter~2]{AmbrGigli:13}.

\begin{Lemma}
  \label{lem:2}
  If $C$ is a compact set in $\cbraces{\ps{2}{\real{d}},W_2}$, then
  the set of martingale measures $D=\cup_{\nu\in C} \Gamma(\nu)$ is
  compact in $(\mathcal{P}_2\cbraces{\real{2d}},W_2)$.
\end{Lemma}

\begin{proof}
  We first show that $D$ is pre-compact. For
  $\nu\in \ps{2}{\real{d}}$, $\gamma \in \Gamma(\nu)$, and a constant
  $K>0$, we deduce that
  \begin{align*}
    \int \abs{x}^2 \ind{\abs{x}\geq K}  d\gamma 
    &= \int \abs{\gamma(y|x)}^2 \ind{\abs{x}\geq K}
      d\gamma  
      \leq \int \abs{y}^2 \ind{\abs{x}\geq K} d\gamma \\
    &\leq \int \abs{y}^2
      \ind{\abs{x} \geq K > \abs{y}^2} d\gamma + \int
      \abs{y}^2 
      \ind{\abs{y}^2 \geq K} d\nu  \\
    &\leq  K \gamma(\abs{x}\geq K) + \int \abs{y}^2
      \ind{\abs{y} \geq \sqrt{K}} d\nu,  \\
    K \gamma(\abs{x}\geq K)
    & \leq \frac1K\int \abs{x}^2 d\gamma =
      \frac1K\int \abs{\gamma(y|x)}^2 d\gamma \leq
      \frac1K \int 
      \abs{y}^2 d\nu, 
  \end{align*}
  and then that
  \begin{align*}
    \frac12 \int \cbraces{\abs{x}^2 + \abs{y}^2}
    \ind{\abs{x} + \abs{y} \geq 
    2K} d\gamma 
    &\leq  \int \abs{x}^2 \ind{\abs{x}\geq K} d\gamma + \int
      \abs{y}^2 \ind{\abs{y}\geq K} d\nu \\
    &\leq \int \abs{y}^2 \cbraces{\frac1K +
      \ind{\abs{y}\geq \sqrt{K}} + \ind{\abs{y}\geq {K}}} d\nu.
  \end{align*}
  The pre-compactness of $D$ follows now from that of $C$ in view
  of~\eqref{eq:11}.

  To show that $D$ is closed, we take a sequence
  $(\gamma_n) \subset D$ converging to $\gamma \in \ps{2}{\real{2d}}$
  under $W_2$. Let $\nu_n$ and $\nu$ be the $y$-marginals of
  $\gamma_n$ and $\gamma$. For every continuous function $f=f(y)$ on
  $\real{d}$ with quadratic growth,
  \begin{displaymath}
    \int f(y) d\nu_n = \int f(y) d\gamma_n \to \int
    f(y) d\gamma =  \int f(y) d\nu.
  \end{displaymath}
  Thus, $\nu_n\to \nu$ in $W_2$. As $C$ is closed, $\nu\in C$.  Let
  $g=g(x)$ be a bounded continuous function on $\real{d}$. From the
  martingale properties of $(\gamma_n)$ we deduce that
  \begin{displaymath}
    \int g(x)(y-x) d\gamma =
    \lim_{n\to \infty} \int g(x)(y-x) d\gamma_n = 0. 
  \end{displaymath}
  It follows that $\gamma(y|x) = x$ or, equivalently, that
  $\gamma \in \Gamma(\nu)$. As $\nu\in C$, we obtain that
  $\gamma\in D$, as required.
\end{proof}

\begin{Lemma}[First-order condition to~\eqref{eq:4}]
  \label{lem:3}
  Let $S\in \msym{d}{m}$, $\nu \in \ps{2}{\real{d}}$, and $\gamma$ be
  an optimal plan for~\eqref{eq:4}. Then
  \begin{equation}
    \label{eq:12}
    \int
    \cbraces{S(x,y) - \frac12 S(x,x)} d\eta 
    \geq \frac12 S(\eta(y),\eta(y))
  \end{equation} 
  for every $\eta\in \mathcal{P}_2(\real{d}\times \real{d})$ such that
  $\supp \eta \subset \supp \gamma$.
\end{Lemma}

\begin{proof}
  We first establish~\eqref{eq:12} for a Borel probability measure
  $\eta$ on $\real{d}\times \real{d}$ having a bounded density
  $V=V(x,y)$ with respect to $\gamma$. We take any
  $\theta \in \ps{2}{\real{d}}$ that is \emph{singular} with respect
  to $\mu(dx) \set \gamma(dx,\real{d})$ and define the product
  probability measure
  $\zeta(dx, dy) \set \theta(dx) \eta(\real{d},dy)$.  For sufficiently
  small $\epsilon>0$, the probability measure
  \begin{displaymath}
    \gamma_{\epsilon} \set \gamma + \epsilon (\zeta - \eta)
  \end{displaymath}
  is well defined and has the same $y$-marginal $\nu$ as $\gamma$.  We
  denote $X_{\epsilon} \set \gamma_{\epsilon}(y|x)$ and observe that
  the law of $(X_{\epsilon}, y)$ under $\gamma_{\epsilon}$ belongs to
  $\Gamma(\nu)$. From the martingale properties of $\gamma_{\epsilon}$
  and $\gamma$ and the optimality of $\gamma$ we deduce that
  \begin{equation}
    \label{eq:13}
    \int S(X_{\epsilon},X_{\epsilon}) d\gamma_{\epsilon} =  \int
    S(X_{\epsilon},y) d\gamma_{\epsilon}\leq \int S(x,y) d\gamma = 
    \int S(x,x) d\gamma. 
  \end{equation}
  
  Let $U(x) \set \gamma(V|x)$, $R(x) \set \gamma(Vy|x)$, and $A$ be a
  separating Borel set in $\real{d}$ for the singular measures $\mu$
  and $\theta$: $\mu(A) = 1-\theta(A) = 1$.  Standard computations
  based on the Bayes formula show that
  \begin{align*}
    X_{\epsilon}(x) =
    \gamma_{\epsilon}(y|x) = \ind{x\in A}\frac{x -
    \epsilon 
    R(x)}{1 - \epsilon U(x)} + \ind{x\not\in A} \eta(y),
    \quad \as{\gamma_{\epsilon}}
  \end{align*}
  Since $0\leq {V}\leq K$ for some constant $K>0$, we deduce that
  $0\leq {U}\leq K$ and $\abs{R} \leq K
  \gamma(\abs{y}|x)$. Conditional Jensen's inequality implies that $R$
  is square integrable under $\gamma$.  It follows that
  \begin{align*}
    \int S(X_{\epsilon},X_{\epsilon}) d\gamma_{\epsilon}
    & = \int \frac{S(x - \epsilon R,x - \epsilon
      R)}{(1-\epsilon U)^2} (1-\epsilon V) d\gamma +
      \epsilon S(\eta(y),\eta(y)) \\
    &= \int \frac{S(x - \epsilon R,x -
      \epsilon R)}{1-\epsilon U} d\gamma + \epsilon S(\eta(y),\eta(y)) \\
    &= \int S(x,x) d\gamma + \epsilon \cbraces{c_1 +
      S(\eta(y),\eta(y))} +
      O(\epsilon^2), 
  \end{align*}
  where
  \begin{align*}
    c_1 & = \int \cbraces{S(x,x) U - 2S(x,R)} d\gamma =
          \int \cbraces{S(x,x)  - 2S(x,y)} V d\gamma \\ 
        & =  \int  \cbraces{S(x,x)  - 2S(x,y)} d\eta.     
  \end{align*}
  In view of~\eqref{eq:13}, $c_1 + S(\eta(y),\eta(y))\leq 0$, which is
  exactly~\eqref{eq:12}.
 
  In the general case, where
  $\eta\in \mathcal{P}_2(\real{d}\times \real{d})$ and
  $\supp{\eta} \subset \supp{\gamma}$, we use an approximation
  argument. Theorem~A.1 in \cite{KramXu:22} yields a sequence
  $(\eta_n)$ of Borel probability measures on $\real{2d}$ having
  bounded densities with respect to $\gamma$ and converging to $\eta$
  under $W_2$. By what we have already proved,
  \begin{displaymath}
    \int \cbraces{S(x,y) - \frac12 S(x,x)} d\eta_n \geq
    \frac{1}{2} S(\eta_n(y),\eta_n(y)), \quad 
    n\geq 1.  
  \end{displaymath}
  Since the integrands are continuous functions with quadratic growth,
  we can pass to the limit as $n\to \infty$ and obtain~\eqref{eq:12}.
\end{proof}

\begin{Lemma}
  \label{lem:4}
  Let $S\in \msym{d}{m}$ and $\nu \in \ps{2}{\real{d}}$. The
  problems~\eqref{eq:4} and~\eqref{eq:6} have solutions and their
  values are identical.
\end{Lemma}
\begin{proof}
  By Lemma~\ref{lem:2}, $\Gamma(\nu)$ is compact in
  $\cbraces{\mathcal{P}_2(\real{2d}), W_2}$. Since the function
  $S=S(x,y)$ is continuous and has quadratic growth, an optimal plan
  $\gamma$ exists.

  Let $B(y)$ be the $y$-section of $\supp{\gamma}$:
  \begin{displaymath}
    B(y) \set \descr{x\in \real{d}}{(x,y)\in
      \supp{\gamma}}, \quad y \in \real{d}, 
  \end{displaymath}
  and denote
  \begin{displaymath}
    f(y) \set \inf_{x\in B(y)}
    \cbraces{S(x,y) - \frac12 S(x,x)}, \quad
    y\in \real{d},  
  \end{displaymath}
  with the usual convention that $f(y) = \infty$ if
  $B(y)=\emptyset$. Let $\zeta$ be a probability measure on $\real{d}$
  with finite support: $\supp{\zeta} = (y_i)_{i=1}^N$. We claim that
  \begin{equation}
    \label{eq:14}
    \zeta(f) = \sum_{i=1}^N f(y_i) \zeta(\braces{y_i})  \geq
    \frac12 S(\zeta(y),\zeta(y)). 
  \end{equation}
  If $\zeta(f) = \infty$, then~\eqref{eq:14} holds
  trivially. Otherwise, for every $\epsilon>0$ there is
  $x_i\in B(y_i)$ such that
  \begin{displaymath}
    f(y_i) > S(x_i,y_i) - \frac12 S(x_i,x_i) -
    \epsilon, \quad i=1,\dots,N. 
  \end{displaymath}
  Let $\eta$ be a probability measure on $\real{2d}$ such that
  \begin{displaymath}
    \eta(\braces{(x_i,y_i)}) =
    \zeta(\braces{y_i}), \quad i=1,\dots,N.  
  \end{displaymath}
  As $(x_i,y_i) \in \supp{\gamma}$, we deduce that
  $\supp{\eta}\subset \supp{\gamma}$ and then get~\eqref{eq:12} from
  Lemma~\ref{lem:3}. It follows that
  \begin{align*}
    \int f(y) d\zeta
    & = \int f(y) d\eta > \int \cbraces{S(x,y) -
      \frac12 S(x,x)} d\eta - \epsilon \\
    & \geq \frac12 S(\eta(y),\eta(y)) - \epsilon  =
      \frac12 S(\zeta(y),\zeta(y)) - \epsilon.   
  \end{align*}
  As $\epsilon$ is any positive number, we obtain~\eqref{eq:14}.

  The minimality property of Fitzpatrick functions from
  Theorem~\ref{th:13} yields $G\in \mset{S}$ such that $f\geq \psi_G$
  or, equivalently, such that
  \begin{displaymath}
    S(x,y) - \frac12 S(x,x) \geq \psi_G(y), \quad
    (x,y) \in \supp{\gamma}.  
  \end{displaymath}
  Accounting for~\eqref{eq:10}, we obtain~\eqref{eq:8} and the
  optimality of $G$.
\end{proof}

\begin{Lemma}
  \label{lem:5}
  Let $S\in \msym{d}{m}$ and $\nu \in \ps{2}{\real{d}}$. For any
  $\gamma \in \Gamma(\nu)$ and $G\in \mset{S}$, the
  conditions~\ref{item:1} and \ref{item:2} of Theorem~\ref{th:1} are
  equivalent.
\end{Lemma}
\begin{proof}
  It is sufficient to show the equivalence of~\ref{item:1} and
  item~\ref{item:5} of Remark~\ref{rem:1}.  By Lemmas~\ref{lem:1}
  and~\ref{lem:4}, $\gamma$ and $G$ are optimal if and only if
  \begin{displaymath}
    \frac12 \int S(x,y) d\gamma = \frac12 \int S(x,x) d\mu
    = \int \psi_G(x) d\mu = \int \psi_G(y) d\nu, 
  \end{displaymath}
  where $\mu$ is the $x$-marginal of $\gamma$.  Theorem~\ref{th:12}
  shows that
  \begin{gather*}
    \psi_G^*(Sx) \geq \psi_G(x) \geq
    \frac12 S(x,x), \quad x\in \real{d},     \\
    x\in G \iff \psi_G(x) = \frac12 S(x,x) \iff \psi^*_G(Sx) = \frac12
    S(x,x).
  \end{gather*}
  The optimality of $\gamma$ and $G$ is thus equivalent to the
  relations:
  \begin{displaymath}
    \mu(G) = 1  \text{ and }  \int
    S(x,y) d\gamma = \int \cbraces{\psi_G(y) +
      \psi^*_G(Sx)} d\gamma.
  \end{displaymath}
  By the construction of conjugate functions,
  \begin{displaymath}
    S(x,y) = \ip{y}{Sx} \leq \psi_G(y) + \psi^*_G(Sx), \quad
    x,y \in \real{d}. 
  \end{displaymath}
  Hence, $\gamma$ and $G$ are optimal if and only if
  \begin{displaymath}
    x \in G \text{ and } S(x,y) = \psi_G(y) +
    \psi^*_G(Sx), \quad \as{\gamma}, 
  \end{displaymath}
  which is equivalent to item~\ref{item:5} of Remark~\ref{rem:1},
  because $G$ is a closed set and $\psi_G$ and $\psi^*_G$ are lower
  semi-continuous functions.
\end{proof}

We denote by $\ri{C}$ the relative interior of a convex set $C$, that
is, the interior with respect to the Euclidean metric structure
restricted to the affine hull of $C$.  If $C$ is a point:
$C = \braces{x_0}$, then $\ri{C} \set \braces{x_0}$.

\begin{Lemma}
  \label{lem:6}
  Let $C$ be a nonempty closed convex set in $\real{d}$ and $x\in
  C$. Then
  \begin{displaymath}
    B_C(x) \set \descr{y\in C}{x + \epsilon(x-y) \in
      C \text{ for some } \epsilon>0}
  \end{displaymath}
  is the largest convex subset of $C$ whose relative interior contains
  $x$. In addition, the set $B_C(x)$ is closed.
\end{Lemma}

\begin{proof}
  The convexity of $B_C(x)$ follows from that of $C$.
  
  Let $A$ be a convex set. Then $x\in \ri A$ if and only if for every
  $y\in A$ there is $\epsilon>0$ such that $x+ \epsilon (x-y) \in
  A$. This fact shows first that $x \in \ri{B_C(x)}$ and then that
  $B_C(x)$ is the largest convex subset of $C$ whose relative interior
  contains $x$.

  Being closed, the set $C$ contains $\cl{B_C(x)}$, the closure of
  $B_C(x)$. Since $x \in \ri{B_C(x)} = \ri{\cl{B_C(x)}}$, we obtain
  that $B_C(x) = \cl{B_C(x)}$.
\end{proof}

\begin{Lemma}
  \label{lem:7}
  Let $C$ be a closed convex set in $\real{d}$,
  $\mu \in \ps{1}{\real{d}}$ be such that $\supp{\mu} \subset C$, and
  denote $v\set \mu(x) = \int xd\mu$.  Then
  \begin{displaymath}
    v \in C \text{ and } \supp{\mu} \subset B_C(v). 
  \end{displaymath}
\end{Lemma}

\begin{proof}
  Let $D\set \clconv{\supp{\mu}}$, the closed convex hull of the
  support of $\mu$. Moving, if necessary, to the affine hull of $D$ we
  can assume that $D$ has a nonempty interior. Clearly, $v\in D$. If
  $v\not\in \interior{D}$, then $v$ belongs to the boundary of
  $D$. Fix a normal vector $u$ to $D$ at $v$. We have that
  \begin{align*}
    \ip{u}{v}  \geq \ip{u}{x}, \quad x \in D.   
  \end{align*}
  As $D$ has a nonempty interior, it is not contained in the
  hyperplane
  \begin{displaymath}
    H \set \descr{x\in \real{d}}{\ip{u}{v} = \ip{u}{x}}.
  \end{displaymath}
  Since $D = \clconv{\supp{\mu}}$, we obtain that
  \begin{displaymath}
    \mu(\descr{x}{\ip{u}{v}  \geq \ip{u}{x}}) = 1 \text{ and }
    \mu(\descr{x}{\ip{u}{v}  > \ip{u}{x}}) > 0 
  \end{displaymath}
  and then arrive to a contradiction:
  \begin{displaymath}
    \ip{u}{v} > \int \ip{u}{x} d\mu = \ip{u}{v}. 
  \end{displaymath}  
  Thus, $v \in \ri{D}$. As $D \subset C$, Lemma~\ref{lem:6} shows that
  $D\subset B_C(v)$.
\end{proof}

The following lemma concludes the proof of Theorem~\ref{th:1}.

\begin{Lemma}
  \label{lem:8}
  Let $S\in \msym{d}{m}$ and $\nu \in \ps{2}{\real{d}}$. For
  $\gamma \in \Gamma(\nu)$ and $G\in \mset{S}$, the
  conditions~\ref{item:2} and \ref{item:3} of Theorem~\ref{th:1} are
  equivalent.
\end{Lemma}

\begin{proof}
  It is sufficient to show the equivalence of~\ref{item:3} and
  item~\ref{item:5} of Remark~\ref{rem:1}.
  
  \ref{item:5} $\implies$ \ref{item:3}: Being a closed set, $G$
  contains the support of $\mu(dx) \set \gamma(dx,
  \real{d})$. Accounting for the martingale property of $\gamma$ and
  the property~\eqref{eq:3} of subdifferentials of convex functions,
  we can choose a regular version of the conditional probability
  $K(x, dy) \set \gamma(dy|x)$ such that
  \begin{gather*}
    \gamma(y|x) = \int y K(x, dy) = x \text{ and } K(x, \partial
    \psi^*_G(Sx)) =1 \text{ for every } x \in G.
  \end{gather*}
  By Theorem~\ref{th:12},
  $G = \descr{x}{x \in \partial \psi_G^*(Sx)}$. Using the notation of
  Lemma~\ref{lem:6}, we can write $Q_G(x)$ as
  \begin{displaymath}
    Q_G(x) = B_{\partial \psi^*_G(Sx)}(x), \quad x
    \in G.  
  \end{displaymath}
  Lemma~\ref{lem:7} now yields that $K(x, Q_G(x)) = 1$ for every
  $x \in G$. The Fubini theorem concludes the argument:
  \begin{displaymath}
    \gamma(\graph{Q_G}) = \int_G \mu(dx) K(x, Q_G(x)) =
    \mu(G) = 1. 
  \end{displaymath}

  \ref{item:3} $\implies$ \ref{item:5}: If $x\in G$, then
  \begin{displaymath}
    y \in Q_G(x) \quad\implies\quad y \in \partial
    \psi^*_G(Sx)\quad\iff \quad Sx \in \partial \psi_G(y). 
  \end{displaymath}
  It follows that
  \begin{displaymath}
    \graph{Q_G}  \subset A \set \descr{(x,y)}{x\in G \text{ and } y \in
      \partial \psi^*_G(Sx)}
  \end{displaymath}
  and then that $\gamma(A) = 1$. By the closedness of $G$ and the
  continuity of the subdifferentials of convex functions, $A$ is a
  closed set.  Hence, $\supp{\gamma} \subset A$, which is
  exactly~\ref{item:5}.
\end{proof}

\section{Properties of optimal plans}
\label{sec:prop-an-optim}

In this section we state some corollaries of Theorem~\ref{th:1}.  We
start with a result showing that, for a \emph{fixed} $G\in \mset{S}$,
the $F_\sigma$-set
\begin{displaymath}
  \graph{Q_G} = \descr{(x, y)}{x \in G,\, y\in Q_G(x)}
\end{displaymath}
is \emph{pointwise minimal} among all Borel sets $B$ in $\real{2d}$
with the property that $\gamma(B) = 1$ for any
$\gamma \in \ps{2}{\real{2d}}$ such that $\gamma$ and $G$ are
solutions to~\eqref{eq:4} and~\eqref{eq:6} for \emph{some}
$\nu \in \ps{2}{\real{d}}$.

\begin{Theorem}
  \label{th:2}
  Let $S\in \msym{d}{m}$ and $G\in \mset{S}$. Then
  \begin{displaymath}
    \graph{Q_G}  =
    \bigcup_{\gamma \in \Gamma_f(G)} \supp{\gamma},
  \end{displaymath}
  where $\Gamma_f(G)$ is the family of probability measures $\gamma$
  on $\real{2d}$ having a finite support and such that $\gamma$ and
  $G$ are solutions to~\eqref{eq:4} and~\eqref{eq:6} for
  $\nu(dy) = \gamma(\real{d}, dy)$.
\end{Theorem}

\begin{proof}
  Item~\ref{item:3} of Theorem~\ref{th:1} shows that
  $\supp{\gamma} \subset \graph{Q_G}$ for every
  $\gamma \in \Gamma_f(G)$.

  Conversely, let $x\in G$ and $u \in Q_G(x)$.  The construction of
  $Q_G(x)$ yields $v \in Q_G(x)$ and $t\in (0,1)$ such that
  $x = t u + (1-t) v$. Setting
  \begin{displaymath}
    \nu(\braces{u}) = \gamma(\braces{(x, u)}) = t, \quad \nu(\braces{v}) =
    \gamma(\braces{(x, v)}) = 1-t, 
  \end{displaymath}
  we deduce that $\gamma \in \Gamma(\nu)$ and its support set
  $\braces{(x,u), (x, v)}$ is a subset of $\graph{Q_G}$. By
  Theorem~\ref{th:1}, $G$ is optimal for~\eqref{eq:6}. Hence,
  $\gamma \in \Gamma_f(G)$.
\end{proof}

For $\mu, \nu \in \ps{2}{\real{d}}$, we write $\mu \cord\nu$ if $\mu$
is dominated by $\nu$ in the convex order: $\mu(f) \leq \nu(f)$ for
every positive convex function $f$.  By \cite[Theorem~2]{Strassen:65},
$\mu \cord \nu$ if and only if $\mu$ is the $x$-marginal of some
$\gamma\in \Gamma(\nu)$.

\begin{Theorem}
  \label{th:3}
  Let $S \in \msym{d}{m}$ and $\nu \in \ps{2}{\real{d}}$. The problem
  \begin{equation}
    \label{eq:15}
    \text{maximize} \quad \frac12 \int S(x,x) d\mu 
    \quad\text{over}\quad \mu \prec_{\text{conv}} \nu,
  \end{equation}
  has a solution.  Let $\mu \cord \nu$, $G\in \mset{S}$, and $\gamma$
  belongs to $\Gamma(\nu)$ and has $\mu$ as its $x$-marginal:
  $\mu(dx) = \gamma(dx,\real{d})$.  Then
  \begin{enumerate}[label = {\rm (\alph{*})}, ref={\rm (\alph{*})}]
  \item \label{item:6} $\mu$ solves~\eqref{eq:15} if and only if
    $\gamma$ is an optimal plan for~\eqref{eq:4}.
  \item \label{item:7} $\mu$ solves~\eqref{eq:15} and $G$ is an
    optimal set for~\eqref{eq:6} if and only if
    \begin{displaymath}
      \int \psi_G(y) d \nu =  \frac12 \int S(x,x) d\mu. 
    \end{displaymath}
    In this case, $\supp{\mu} \subset G$.
  \end{enumerate}
\end{Theorem}
\begin{proof}
  From the martingale property of $\gamma$ we deduce that
  \begin{align*}
    \int S(x,y) d\gamma = \int S(x,x) d\gamma = \int S(x,x) d\mu
  \end{align*}
  and then get~\ref{item:6}. We obtain~\ref{item:7} from the identity
  above, Theorem~\ref{th:1}, and the fact that $G$ is closed.
\end{proof}

\begin{Remark}
  \label{rem:2}
  Assume that the problem \eqref{eq:4} admits a \emph{unique} optimal
  plan $\gamma\in \Gamma(\nu)$. According to Theorem \ref{th:3},
  $\mu(dx)\set\gamma(dx,\real{d})$ is the unique solution for
  \eqref{eq:15} and $\gamma$ is the \emph{unique martingale coupling}
  between $\mu$ and $\nu$:
  \begin{displaymath}
    \rho\in \Gamma (\nu), \ \rho(dx,\real{d})=\mu(dx)\implies
    \rho=\gamma.
  \end{displaymath}
  The latter property is quite special. Usually, probability measures
  $\mu$ and $\nu$ such that $\mu\cord\nu$ can be coupled by an
  infinite number of martingale measures.
\end{Remark}

We now show that an optimal plan $\gamma$ for~\eqref{eq:4} is a
classical optimal $\mathcal{L}_2$-coupling between its $Sx$ and $y$
marginals.

\begin{Theorem}
  \label{th:4}
  Let $S\in \msym{d}{m}$, $\nu \in \ps{2}{\real{d}}$, $\gamma$ be an
  optimal plan for~\eqref{eq:4}, and $\mu$ be the $x$-marginal of
  $\gamma$. Then $\gamma$ is a solution of the (unconstrained) optimal
  transport problem:
  \begin{displaymath}
    \text{maximize} \quad \int
    S(x,y) d\pi = \int \ip{Sx}{y}d\pi \quad 
    \text{over} \quad \pi\in \Pi(\mu,\nu),    
  \end{displaymath}
  where $\Pi(\mu, \nu)$ is the family of Borel probability measures on
  $\real{d}\times \real{d}$ with $x$-marginal $\mu$ and $y$-marginal
  $\nu$.
\end{Theorem}

\begin{proof}
  Theorem~\ref{th:1} and Remark~\ref{rem:1} yield $G\in \mset{S}$ such
  that $Sx \in \partial \psi_G(y)$ as soon as
  $(x,y)\in \supp{\gamma}$.  By the Legendre-Fenchel inequality,
  \begin{displaymath}
    \psi_G(y) + \psi^*_G(Sx) \geq \ip{y}{Sx} = S(x,y), \quad
    x,y\in \real{d},  
  \end{displaymath}
  while by the properties~\eqref{eq:3} of subdifferentials of convex
  functions,
  \begin{displaymath}
    \psi_G(y) + \psi^*_G(Sx) = S(x,y), \quad
    (x,y)\in \supp{\gamma}.  
  \end{displaymath}
  Thus, for every $\pi\in \Pi(\mu,\nu)$,
  \begin{align*}
    \int S(x,y)d\pi
    & \leq \int
      \cbraces{\psi_G(y) + \psi^*_G(Sx)}d\pi = \int \psi_G(y)
      d\nu + \int
      \psi^*_G(Sx) d\mu
    \\  & = \int
          \cbraces{\psi_G(y) + \psi^*_G(Sx)}d\gamma 
          =  \int S(x,y) d\gamma, 
  \end{align*}
  as claimed.
\end{proof}

We conclude the section with easy to check sufficient conditions for
the uniqueness of the solution to the dual problem~\eqref{eq:6}. The
uniqueness of the optimal plan for the primal problem~\eqref{eq:4} is
studied in the follow-up paper~\cite{KramSirb:23} and relies on the
classification of the singularities of the projection $P_G$
in~\cite{KramSIrb:22b}.

\begin{Theorem}
  \label{th:5}
  Let $S\in \msym{d}{m}$ and $\nu \in \ps{2}{\real{d}}$.  If
  $\supp \nu$ is bounded \emph{or} convex, then the dual problem
  \eqref{eq:6} has a unique solution over $\supp \nu$ in the sense
  that
  \begin{displaymath}
    \psi _{\widetilde{G}}(y)=\psi _G(y), \ y \in \supp \nu, 
    \text{~and~} G\cap \supp{\nu} = \widetilde G \cap \supp{\nu},
  \end{displaymath}
  for any two optimal sets $G$ and $\widetilde{G}$.
  
  In particular, if $\supp \nu$ is convex and there exists an optimal
  $G$ contained in $\supp {\nu}$, then $G$ is the unique optimal
  set. This condition holds trivially if $\supp{\nu} = \real{d}$.
\end{Theorem}
\begin{proof}
  Consider two optimal sets $G$ and $\widetilde{G}$ for \eqref{eq:6}
  and let $\gamma$ be an optimal plan for~\eqref{eq:4}.
  Item~\ref{item:4} in Remark~\ref{rem:1} shows that
  \begin{displaymath}
    \psi_G(y) = S(x,y) - \frac12 S(x,x) = \psi_{\widetilde G}(y), \quad
    (x,y) \in \supp{\gamma}. 
  \end{displaymath}
  Denote by $F$ the projection of $\supp{\gamma}$ on the
  $y$-coordinate.  While, in the general case, $F$ may not be closed,
  we know that $\cl F=\supp \nu$.  From the relation above, it follows
  that the Fitzpatrick functions $\psi_G$ and $\psi_{\widetilde G}$
  coincide and are finite over $F$.
   
  If $\supp \nu$ is bounded, the martingale property implies that the
  $x$-projection of $\supp \gamma$ is bounded as well. A simple
  compactness argument shows that $F$ is closed, so $\supp \nu =F$.
   
  If $C\set \supp \nu$ is convex, we have $F\subset C=\cl F$. By
  restricting to the affine space generated by $C$, we can assume that
  $C$ has nonempty interior.  From $C=\cl F$ we obtain that
  $\interior C=\interior \conv F$.  The closed convex functions
  $\psi _G$ and $\psi _{\widetilde{G}}$ are finite on $F$, so they are
  finite on $\conv F$ and finite and continuous on $\interior C$.  We
  know that $\psi_G$ and $\psi _{\widetilde{G}}$ coincide on $F$.
  Using $\cl F=C$, we obtain that they coincide on $\interior C$.
  Finally, the equality of the closed convex functions $\psi _G$ and
  $\psi _{\widetilde{G}}$ extends to $\cl \interior C=C$.

  In both situations, when either $\supp \nu$ is bounded or
  $\supp \nu$ is convex, it was proved that
  $\psi _G=\psi _{\widetilde{G}}$ over $\supp \nu$.  Now, by the
  properties of Fitzpatrick functions from Theorem~\ref{th:12},
  \begin{align*}
    G \cap \supp \nu  = \descr{x\in
    \supp \nu}{\psi_G(x) = \psi_{\widetilde{G}}(x) = \frac12 S(x,x)} = 
    \widetilde{G} \cap \supp \nu. 
  \end{align*}
  If $\supp \nu$ is convex and there exists an optimal set
  $G \subset \supp{\nu}$, then
  $G = \widetilde{G} \cap \supp{\nu} \subset \widetilde{G}$ for any
  other optimal set $\widetilde{G}$. The maximality of $G$ yields that
  $G=\widetilde{G}$.
\end{proof}

\section{Linear solutions}
\label{sec:linear}

The main result of this section, Theorem~\ref{th:9}, states the
necessary and sufficient conditions for a solution to the dual
problem~\eqref{eq:6} to be an affine subset of $\real{d}$. In this
case, an optimal plan for~\eqref{eq:4} is unique and has an explicit
linear form.

We start with some results from linear algebra. By $I$ we denote the
identity matrix and by $V^T$ the transpose of a matrix $V$.  Let $A$
be a symmetric $d\times d$ matrix. We write $A>0$ and say that $A$ is
\emph{positive definite} if $\ip{x}{Ax} > 0$ for every
$x\in \real{d}\setminus 0$.  We write $A\geq 0$ and say that $A$ is
\emph{positive semi-definite} if $\ip{x}{Ax} \geq 0$ for every
$x\in \real{d}$.  By $A\leq 0$ we mean that $-A\geq 0$. A similar
convention holds for $A<0$. Notice that $A>0$ if and only if
$A\in \msym{d}{d}$ and $A<0$ if and only if $A\in \msym{d}{0}$.

\begin{Theorem}
  \label{th:6}
  Let $S\in \msym{d}{m}$ and $\Sigma$ be a symmetric positive definite
  $d\times d$ matrix. There exist unique symmetric positive
  semi-definite $d\times d$ matrices $Q$ and $R$ such that
  \begin{equation}
    \label{eq:16}
    \Sigma =Q+R, \quad QSQ \geq 0, \quad RSR \leq 0, \quad QSR =
    0. 
  \end{equation}
  In this case,
  \begin{gather*}
    \rank{Q} = m, \quad \rank{R} = d-m, \quad QSQ-RSR>0, \\
    Q\Sigma^{-1}Q = Q, \quad R\Sigma^{-1}R = R, \quad Q\Sigma^{-1}R =
    0.
  \end{gather*}
\end{Theorem}
\begin{proof}
  The classical diagonalization theorem of linear algebra,
  \cite[Theorem~9, p.~314]{Gant-1:98}, yields a $d\times d$ matrix $V$
  of full rank such that
  \begin{displaymath}
    V^T S V = \Lambda, \quad V^T \Sigma^{-1} V = I, 
  \end{displaymath}
  where $\Lambda$ is a diagonal matrix:
  $\Lambda_{ij} = \delta_{ij} \lambda_i$, whose first $m$ diagonal
  elements are strictly positive and the remaining $d-m$ diagonal
  elements are strictly negative.

  Observe now that the symmetric positive semi-definite matrices $Q$
  and $R$ satisfy~\eqref{eq:16} if and only if
  \begin{displaymath}
    Q = VAV^T, \quad R = VBV^T, 
  \end{displaymath}
  where $A$ and $B$ are positive semi-definite matrices such that
  \begin{displaymath}
    I = A + B, \quad A\Lambda A \geq 0, \quad B \Lambda B
    \leq 0, \quad A \Lambda B = 0.   
  \end{displaymath}
  We claim that the only such $A$ and $B$ are diagonal and given by
  \begin{displaymath}
    A_{ij} = \delta_{ij} \ind{i\leq m}, \quad B_{ij} = \delta_{ij}
    \ind{i>m}, \quad i,j\in \braces{1,\dots,d}. 
  \end{displaymath}
  Checking that the above matrices $A$ and $B$ are solutions is
  immediate.  The uniqueness is verified as follows.

  From $A\Lambda B=0$ and $A+B=I$ we obtain that
  $\Lambda A =A\Lambda A$ and $\Lambda B=B\Lambda B$. In particular,
  the matrix $\Lambda A$ is symmetric:
  \begin{displaymath}
    \lambda_i A_{ij} = \lambda_j A_{ji}
    \quad  i,j\in \braces{1,\dots,d}.
  \end{displaymath}
  Because $A$ is symmetric and $\lambda_i>0$ for $i\leq m$ and
  $\lambda_i<0$ for $i>m$, we deduce that
  \begin{displaymath}
    A_{ij} = A_{ji} = 0, \quad i\leq m< j.
  \end{displaymath}
  Since $\Lambda A = A\Lambda A\geq 0$, we have that
  \begin{displaymath}
    0\leq \sum_{ij} \lambda_i A_{ij} x_ix_j
    = \sum_{i} \lambda_i A_{ii} x_i^2 +  \sum_{i<j} (\lambda_i +
    \lambda_j) A_{ij} x_ix_j,  \quad x\in \real{d}.  
  \end{displaymath}
  Recalling that $A\geq 0$, taking distinct
  $k,l\in \braces{m+1,\dots,d}$, and choosing $x_i = \delta_{ik}$ and
  $x_i = \delta_{ik} + \delta_{il}$ we obtain that
  \begin{displaymath}
    A_{kl} =0, \quad k,l>m.
  \end{displaymath}
  Thus the matrix $A$ has possible nonzero entries $A_{ij}$ only for
  $i,j\leq m$.  Using $\Lambda B=B\Lambda B\leq 0$ and $B\geq 0$ we go
  over identical arguments to obtain that the matrix $B$ has possible
  nonzero entries $B_{ij}$ only for $i,j> m$.  The uniqueness follows
  now from $A+B=I$.

  We have shown that symmetric positive semi-definite matrices $Q$ and
  $R$ satisfying \eqref{eq:16} exist and are unique.  The remaining
  properties of $Q$ and $R$ follow from their counterparts for $A$ and
  $B$:
  \begin{gather*}
    \rank{A}= m, \quad \rank{B} = d-m, \quad A\Lambda A - B\Lambda B >
    0, \\
    AIA =A, \quad BIB = B, \quad AIB = 0.
  \end{gather*}
\end{proof}

We recall that a $d\times d$ matrix $P$ is called \emph{idempotent} if
$P^2 = P$.

\begin{Theorem}
  \label{th:7}
  Let $S\in \msym{d}{m}$, $\Sigma$ be a symmetric positive definite
  $d\times d$ matrix, and $Q$ and $R$ be the symmetric positive
  semi-definite matrices satisfying~\eqref{eq:16}. Then
  $P \set Q\Sigma^{-1}$ is the unique idempotent $d\times d$ matrix
  such that
  \begin{equation}
    \label{eq:17}
    \begin{split}
      \text{the matrices~} P\Sigma \text{~and~}& SP \text{~are symmetric,} \\
      P\Sigma \geq 0, \quad (I-P)\Sigma \geq 0, \quad & SP \geq 0,
      \quad S(I-P) \leq 0.
    \end{split}
  \end{equation}
  In this case,
  \begin{displaymath}
    SP - S(I-P) = S(2P-I) > 0. 
  \end{displaymath}
\end{Theorem}

\begin{proof}
  We have that $P\Sigma = Q \geq 0$, $(I-P)\Sigma = R\geq 0$, and
  \begin{equation}
    \label{eq:18}
    \begin{split}
      (I-P)^T SP &= \Sigma^{-1} RSQ \Sigma^{-1} = 0,\\
      SP & = P^TSP = \Sigma^{-1} QSQ \Sigma^{-1} \geq 0, \\
      S(I-P) & = (I-P)^TS(I-P) = \Sigma^{-1} RSR \Sigma^{-1} \leq 0.
    \end{split}
  \end{equation}
  By Theorem~\ref{th:6}, $Q\Sigma^{-1}Q = Q$, so
  \begin{displaymath}
    P^2 = Q\Sigma ^{-1}Q\Sigma ^{-1}= Q\Sigma ^{-1}=P.
  \end{displaymath}
  Thus, $P$ is an idempotent matrix solving~\eqref{eq:17}. Moreover,
  by Theorem~\ref{th:6},
  \begin{displaymath}
    SP - S(I-P)  =  \Sigma^{-1} (QSQ - RSR)
    \Sigma^{-1} >0.  
  \end{displaymath}

  Conversely, if $P'$ is an idempotent matrix solving~\eqref{eq:17},
  then the relations in~\eqref{eq:18} (applied in the reverse order)
  show that $Q' \set P'\Sigma$ and $R'\set (I-P')\Sigma$ are symmetric
  positive semi-definite matrices satisfying~\eqref{eq:16}. The
  uniqueness part of Theorem~\ref{th:6} yields that $Q'=Q$. Hence,
  $P=P'$.
\end{proof}

We conclude the introductory part of the section with a description of
all affine maximal \emph{strictly} (in the sense of~\eqref{eq:19})
$S$-monotone sets.

\begin{Theorem}
  \label{th:8}
  Let $S\in \msym{d}{m}$. The following conditions are equivalent:
  \begin{enumerate}[label = {\rm (\alph{*})}, ref={\rm (\alph{*})}]
  \item \label{item:8} $G$ is an affine subspace of $\real{d}$,
    $G \in \mset{S}$, and
    \begin{equation}
      \label{eq:19}
      S(x-y,x-y) >0, \quad x,y \in G, \; x\not = y. 
    \end{equation}
  \item \label{item:9} $G = \descr{x_0 + P(x-x_0)}{x\in \real{d}}$,
    where $x_0 \in \real{d}$ and $P$ is a $d\times d$ idempotent
    matrix such that the matrix $SP$ is symmetric and
    \begin{equation}
      \label{eq:20}
      SP\geq 0, \quad S(I-P)\leq 0, \quad S(2P-I) > 0. 
    \end{equation}
  \end{enumerate}
  The projection on $G$ in the $S$-space has the form:
  \begin{displaymath}
    P_G(y) \set \argmin_{x\in G} S(y-x,y-x) = x_0 + P(y-x_0), \quad
    y\in \real{d}.  
  \end{displaymath}
\end{Theorem}

\begin{proof}
  Without restriction of generality we can assume that $G$ is a linear
  subspace, that is, it contains $0$. We then take $x_0 = 0$.

  \ref{item:8} $\implies$ \ref{item:9}: Standard arguments show that
  for every $y\in \real{d}$ the function $x\to S(y-x,y-x)$ attains the
  minimum on $G$ at a unique $P_G(y)$ such that
  \begin{equation}
    \label{eq:21}
    S(y-P_G(y),x) = 0, \quad x\in G. 
  \end{equation}
  We first observe that $P_G(y)=y$ for $y\in G$.  Then, since $G$ is a
  linear subspace, we deduce that $P_G$ is a linear function mapping
  $\real{d}$ onto $G$:
  \begin{displaymath}
    P_G(y) = Py, \quad y\in \real{d},
  \end{displaymath}
  for a unique $d\times d$-matrix $P$.  As
  \begin{displaymath}
    P^2 y = P_G(P_G(y)) = P_G(y) = Py, \quad y\in \real{d}, 
  \end{displaymath}
  $P$ is an idempotent matrix. Since
  \begin{displaymath}
    G = \range{P_G} \set \descr{P_G(y)}{y\in \real{d}} =
    \descr{Py}{y\in \real{d}}, 
  \end{displaymath}
  we can write~\eqref{eq:21} as
  \begin{displaymath}
    S(y-Py, Px)  = \ip{(I-P)y}{SPx}= 0, \quad x,y\in \real{d}, 
  \end{displaymath}
  or, equivalently, as
  \begin{displaymath}
    0 = (I-P)^T S P = SP - P^TSP. 
  \end{displaymath}
  In particular, $SP$ is a symmetric matrix. As $\range P=G$, we
  deduce from~\eqref{eq:19} that $SP = P^T S P \geq 0$.  Since
  $G\in \mset{S}$, we obtain that
  \begin{displaymath}
    S(y-Py, y-Py) = \min_{x\in G} S(y-x,y-x) < 0, \quad y\in
    \real{d} \setminus G.  
  \end{displaymath}
  As $y = Py$, $y\in G$, and $SP = P^TS=P^TSP$, we deduce that
  \begin{displaymath}
    0 \geq (I-P)^T S (I-P) = S(I-P). 
  \end{displaymath}
  Accounting for~\eqref{eq:19} we also obtain that
  \begin{displaymath}
    0 < P^T S P - (I-P)^T S (I-P) = SP - S(I-P) = S(2P-I).  
  \end{displaymath}

  \ref{item:9} $\implies$ \ref{item:8}: Recall that $x_0 = 0$ and
  thus,
  \begin{displaymath}
    G \set \range{P} = \descr{Py}{y\in \real{d}}.
  \end{displaymath}
  To finish the proof of the theorem, we have to show that
  $G \in \mset{S}$ and that~\eqref{eq:19} holds.  If $x\in G$, then
  $x = (2P-I)x$. Since the matrix $U \set S(2P-I)$ is positive
  definite, we obtain that
  \begin{displaymath}
    S(x,x) = S(x,(2P-I)x) = \ip{x}{Ux} > 0, \quad x\in
    G\setminus 0.   
  \end{displaymath}
  Hence, $G$ is an $S$-monotone set satisfying~\eqref{eq:19}.
  
  As $(2P-I)(I-P)= -(I-P)$, we obtain that
  \begin{displaymath}
    (I-P)^TU(I-P) = - (I-P)^TS(I-P). 
  \end{displaymath}
  If $y\not\in G$, then $z\set y - Py = (I - P)y \not=0$.  It follows
  that
  \begin{displaymath}
    S(z,z) = \ip{(I-P)y}{S(I-P)y} = - \ip{(I-P)y}{U(I-P)y} = -
    \ip{z}{Uz} < 0.
  \end{displaymath}
  Since $Py\in G$, this proves the maximality of $G$.
\end{proof}

For simplicity of notation, in Theorems~\ref{th:9} and~\ref{th:10} we
use a probabilistic setup. We start with a $d$-dimensional random
variable $Y$ on a probability space
$(\Omega, \mathcal{F}, \mathbb{P})$ having a finite second moment:
$Y\in \lsp{2}{d}$. We denote $\nu \set \law{Y}$, the law of $Y$.  Let
$X \in \lsp{2}{d}$. Clearly, $\law{X,Y} \in \Gamma(\nu)$ if and only
if $X = \cEP{Y}{X}$. As usual, we interpret all relations between
random variables in the $\as{\mathbb{P}}$ sense.

\begin{Theorem}
  \label{th:9}
  Let $S\in \msym{d}{m}$, $Y \in \mathcal{L}_2(\real{d})$, assume that
  the covariance matrix of $Y$ is positive definite:
  \begin{displaymath}
    \Sigma \set \EP{(Y - \EP{Y})(Y - \EP{Y})^T} > 0,  
  \end{displaymath}
  and denote $\nu \set \law{Y}$.  Let $P$ be the unique idempotent
  $d\times d$ matrix satisfying~\eqref{eq:17}. The following
  conditions are equivalent:
  \begin{enumerate}[label = {\rm (\alph{*})}, ref={\rm (\alph{*})}]
  \item \label{item:10} An optimal set $G$ for~\eqref{eq:6} is an
    affine subspace of $\real{d}$.
  \item \label{item:11} The random variable
    $X\set \EP{Y} + P(Y-\EP{Y})$ has the martingale property:
    $\cEP{Y}{X} = X$.
  \end{enumerate}
  In this case, the law of $(X,Y)$ is the unique optimal plan
  for~\eqref{eq:4} and
  \begin{equation}
    \label{eq:22}
    G \set \descr{\EP{Y} + P(x-\EP{Y})}{x\in \real{d}}
  \end{equation}
  is the unique affine subspace of $\real{d}$ which is an optimal set
  for~\eqref{eq:6}.

  If $G = \supp{\mu}$ for $\mu \set \law{X}$, then $G$ is the unique
  optimal set.
\end{Theorem}

\begin{proof}
  \ref{item:10} $\implies$ \ref{item:11}: Let us
  show~\eqref{eq:19}. If this condition fails, then there are distinct
  $x_0,x_1 \in G$ such that $S(x_1-x_0,x_1-x_0) = 0$. As $G$ is an
  affine subspace,
  \begin{displaymath}
    x(t) \set x_0 + t(x_1-x_0) \in G, \quad t\in \real{},  
  \end{displaymath}
  and we obtain the lower bound for the Fitzpatrick function $\psi_G$:
  \begin{align*}
    \psi_G(y) & = \sup_{x\in G}\cbraces{S(x,y) - \frac12 S(x,x)} \geq
                \sup_{t\in \real{}} \cbraces{S(x(t),y) - \frac12 S(x(t),x(t))} \\
              &= S(x_0,y) - \frac12 S(x_0,x_0) + \sup_{t\in \real{}} \cbraces{t
                S(y-x_0,x_1-x_0)}. 
  \end{align*}
  It follows that the domain of $\psi_G$ is contained in the
  hyperplane
  \begin{displaymath}
    H \set \descr{y\in \real{d}}{S(y-x_0,x_1-x_0) = 0} 
  \end{displaymath}
  of dimension $d-1$. As
  $\EP{\psi_G(Y)} = \int \psi_G(y) d\nu < \infty$, we deduce that
  $Y\in H$. In particular, the random variables $(Y^n)_{n=1,\dots,d}$
  are linearly dependent, in contradiction with the assumption that
  their covariance matrix $\Sigma>0$.
  
  Given~\eqref{eq:19}, Theorem~\ref{th:8} states that
  \begin{displaymath}
    G = \descr{x_0 + \widetilde{P}(x-x_0)}{x\in \real{d}} 
  \end{displaymath}
  for the idempotent matrix $\widetilde{P}$ satisfying~\eqref{eq:20}
  and some $x_0\in \real{d}$. It also shows that the projection on $G$
  in the $S$-space has the form:
  \begin{displaymath}
    P_G(y) = x_0 + \widetilde{P}(y-x_0), \quad y\in \real{d}.  
  \end{displaymath}

  By Theorem~\ref{th:1}, an optimal plan for~\eqref{eq:4}
  exists. Moreover, as $G$ is an optimal set for \eqref{eq:6}, every
  optimal plan $\gamma$ is supported on the graph
  $\descr{(P_G(y),y)}{y\in \real{d}}$. Since the projection $P_G$ is
  single-valued, such $\gamma$ is unique and given by the joint law of
  $(X,Y)$, where
  \begin{displaymath}
    X \set P_G(Y) = x_0 + \widetilde{P}(Y-x_0). 
  \end{displaymath}
  In particular, $X = \cEP{Y}{X}$ and
  \begin{displaymath}
    \EP{Y} = \EP{X} = x_0 + \widetilde{P}(\EP{Y}-x_0). 
  \end{displaymath}
  It follows that
  \begin{displaymath}
    G = \descr{x_0 + \widetilde{P}(x-x_0)}{x\in \real{d}} = \descr{\EP{Y} +
      \widetilde{P}(x-\EP{Y})}{x\in \real{d}}.  
  \end{displaymath}
  Thus, we can take $x_0 = \EP{Y}$.
  
  As $X = \cEP{Y}{X}$, we have that
  \begin{align*}
    0 & = \EP{(Y-X)(X-\EP{X})^T} = (I-\widetilde{P}) \Sigma
        \widetilde{P}^T = \Sigma \widetilde{P}^T - \widetilde{P} \Sigma
        \widetilde{P}^T \\
      & = (I-\widetilde{P})\Sigma - (I-\widetilde{P}) \Sigma (I-\widetilde{P})^T. 
  \end{align*}
  It follows that $\Sigma \widetilde{P}^T$ and
  $(I-\widetilde{P})\Sigma$ are symmetric positive semi-definite:
  \begin{align*}
    \Sigma \widetilde{P}^T &=\widetilde{P}\Sigma \widetilde{P}^T
                             =    \EP{(X - \EP{X})(X- \EP{X})^T} \geq 0,  \\
    (I-\widetilde{P})\Sigma &= (I-\widetilde{P}) \Sigma (I-\widetilde{P})^T
                              = \EP{(Y - X)(Y - X)^T} \geq 0.
  \end{align*}   
  Thus, $\widetilde{P}$ is a solution
  of~\eqref{eq:17}. Theorem~\ref{th:7} shows that $P =
  \widetilde{P}$. Hence, $G$ is given by~\eqref{eq:22}.
  
  \ref{item:11} $\implies$ \ref{item:10}: By Theorems~\ref{th:7}
  and~\ref{th:8}, the set $G$ in~\eqref{eq:22} belongs to $\mset{S}$
  and $X = P_G(Y)$.  Theorem~\ref{th:1} shows that $G$ is an optimal
  set for~\eqref{eq:6} and $\gamma \set \law{X,Y}$ is the unique
  optimal plan for~\eqref{eq:4}.

  If $\widetilde{G}$ is another optimal set (not necessarily affine),
  then item~\ref{item:3} of Theorem~\ref{th:1} yields that
  \begin{displaymath}
    \supp{\mu} \subset G \cap \widetilde{G}.
  \end{displaymath}
  If $\supp{\mu} = G$, then $G\subset \widetilde{G}$ and the
  maximality of $G$ implies that $G=\widetilde{G}$.
\end{proof}

Assume the setup of Theorem~\ref{th:9} and that $\EP{Y} = 0$.
Item~\ref{item:11} of Theorem~\ref{th:9} can be equivalently stated as
\begin{displaymath}
  \cEP{(I-P)Y}{PY} = 0.
\end{displaymath}
Since $(I-P) \Sigma P^T = 0$, item~\ref{item:11} holds if
\begin{displaymath}
  A\Sigma B^T = 0 \implies \cEP{AY}{BY} = 0,
\end{displaymath}
for all $d\times d$ matrices $A$ and $B$. This is the case, for
instance, if $Y$ has a Gaussian or, more generally, elliptically
contoured distribution with mean zero and covariance matrix $\Sigma$.

\begin{Theorem}
  \label{th:10}
  Let $S\in \msym{d}{m}$ and $Y$ be a $d$-dimensional Gaussian random
  variable with mean zero and the covariance matrix $\Sigma>0$. Let
  $P$ be the unique idempotent $d\times d$ matrix
  satisfying~\eqref{eq:17}.  Then
  \begin{displaymath}
    X=PY, \quad Z= (I-P)Y,
  \end{displaymath}
  are the unique independent $d$-dimensional Gaussian random variables
  with mean zero such that
  \begin{equation}
    \label{eq:23}
    \begin{split}
      & Y=X+Z, \\
      S(X,X) \geq 0, \quad & S(Z,Z)\leq 0, \quad S(X,Z) =0.
    \end{split}
  \end{equation}
  The law of $(X,Y)$ is the unique optimal plan for~\eqref{eq:4} and
  \begin{displaymath}
    \range{P} \set \descr{Px}{x\in \real{d}} 
  \end{displaymath}
  is the unique optimal set for~\eqref{eq:6}, where $\nu\set \law{Y}$.
\end{Theorem}
\begin{proof}
  Clearly, $X$ and $Z$ are $d$-dimensional Gaussian random variables
  with mean zero. Their covariance matrices are given by
  \begin{align*}
    \EP{XX^T} & = P \EP{YY^T} P^T = P\Sigma P^T = P\Sigma, \\
    \EP{ZZ^T} & =  (I-P)\Sigma (I-P)^T =
                (I-P)\Sigma, \\
    \EP{XZ^T} & =  P\Sigma (I-P)^T = 0. 
  \end{align*}
  The last identity implies that $X$ and $Z$ are independent random
  variables. The relations~\eqref{eq:23} readily follow from the
  properties~\eqref{eq:17} of the idempotent matrix $P$.

  Let $X'$ and $Z'$ be independent $d$-dimensional Gaussian random
  variables with mean zero, the covariance matrices $Q'$ and $R'$,
  respectively, and such that
  \begin{gather*}
    Y = X' + Z', \\
    S(X',X') \geq 0, \quad S(Z',Z')\leq 0, \quad S(X',Z') =0.
  \end{gather*}
  As $X'$ and $Z'$ are Gaussian random variables, their laws are
  supported by the ranges of the covariance matrices.  Accounting for
  the independence of $X'$ and $Z'$ we deduce that
  \begin{gather*}
    Q' + R' = \Sigma, \\
    S(x,x) \geq 0, \; S(z,z) \leq 0, \; S(x,z) = 0, \quad x\in
    \range{Q'}, \; z\in \range{R'},
  \end{gather*}
  and then that
  \begin{displaymath}
    Q'SQ' \geq 0, \quad  R'SR' \leq 0, \quad Q'SR' = 0. 
  \end{displaymath}
  Theorem~\ref{th:7} shows that $P = Q' \Sigma^{-1}$. Then
  $Q' = P\Sigma$ and $R' = (I-P)\Sigma$ are also the covariance
  matrices of $X$ and $Z$, respectively.  It follows that
  $\law{X',Z'} = \law{X,Z}$, that
  \begin{displaymath}
    \law{X',Y} = \law{X,Y} = \law{PY, Y}, 
  \end{displaymath}
  and then that $X'=PY=X$ and $Z' = Y-X' = Y-X = Z$.

  Finally, Theorem~\ref{th:9} yields the optimality and uniqueness
  properties of $\gamma \set \law{X,Y}$ and $G\set \range{P}$.
\end{proof}

\appendix

\section{Fitzpatrick functions in the $S$-space}
\label{sec:basic-prop-fitzp}

For $S\in \msym{d}{m}$, the family of symmetric $d\times d$-matrices
of full rank with $m$ positive eigenvalues, we recall the notation
\begin{displaymath}
  S(x,y) \set \ip{x}{Sy} = \ip{Sx}{y}, \quad x,y\in \real{d}, 
\end{displaymath}
for the bilinear form and $\mset{S}$ for the family of maximal
$S$-monotone sets. Clearly, $S^{-1}$, the inverse matrix to $S$,
belongs to $\msym{d}{m}$.

For $S \in \msym{2m}m$ from Example \ref{ex:3}, many results of this
appendix can be found in \cite{Fitzpatrick:88} and
\cite{PenotZalin:05}. The case of general symmetric $S$ has been
studied in \cite{Simons:07} and~\cite{Penot:09}.

\begin{Theorem}
  \label{th:11}
  Let $S\in \msym{d}{m}$ and $\map{f}{\real{d}}{\realext}$ be a closed
  convex function such that
  \begin{equation}
    \label{eq:24}
    \min(f(x), f^*(Sx)) \geq \frac12 S(x,x), \quad x\in \real{d}.   
  \end{equation}
  Then the set $G \set \descr{x \in \real{d}}{Sx\in \partial f(x)}$
  belongs to $\mset{S}$ and has equivalent descriptions:
  \begin{equation}
    \label{eq:25}
    G = \descr{x \in \real{d}}{f(x) = \frac12 S(x,x)} = \descr{x
      \in \real{d}}{f^*(Sx) = \frac12 S(x,x)}. 
  \end{equation}
\end{Theorem}
\begin{proof}
  By the properties~\eqref{eq:3} of subdifferentials of convex
  functions,
  \begin{align*}
    Sx \in \partial f(x) \iff x\in \partial
    f^*(Sx)  \iff f(x) + f^*(Sx) =  S(x,x). 
  \end{align*}
  The latter representation and~\eqref{eq:24} imply that
  \begin{displaymath}
    G = \descr{x \in \real{d}}{f(x) = 
      f^*(Sx) = \frac12 S(x,x)}.
  \end{displaymath}
  
  Let $u\in \real{d}$. If $f(x) = \frac12 S(x,x)$, then
  inequality~\eqref{eq:24} yields that
  \begin{align*}
    \lim_{t\downarrow 0}\frac{1}t \cbraces{f(x+tu) -
    f(x)} & \geq \lim_{t\downarrow 0} \frac1{2t} \cbraces{S(x+tu,x+tu) -
            S(x,x)} \\ & = S(x,u) = \ip{Sx}{u}.
  \end{align*}
  It follows that $Sx \in \partial f(x)$, proving the first equality
  in~\eqref{eq:25}.

  Similarly, if $f^*(Sx) = \frac12 S(x,x)$, then
  \begin{align*}
    \lim_{t\downarrow 0}\frac{1}t \cbraces{f^*(Sx+tu) -
    f^*(Sx)} & \geq \lim_{t\downarrow 0} \frac1{2t}
               \cbraces{S(x+tS^{-1}u,x+tS^{-1}u) - S(x,x)}
    \\ & = S(x,S^{-1}u) = \ip{x}{u}. 
  \end{align*}
  It follows that $x \in \partial f^*(Sx)$ or equivalently, that
  $Sx \in \partial f(x)$, proving the second equality
  in~\eqref{eq:25}.

  If $x,y\in G$, then $(x,Sx)$ and $(y, Sy)$ belong to the graph of
  $\partial f$, which is a classical monotone set in $\real{2d}$:
  \begin{displaymath}
    S(x-y,x-y) = \ip{x-y}{Sx-Sy} \geq 0. 
  \end{displaymath}
  Hence, $G$ is $S$-monotone.

  To show that $G$ is maximal, we assume first that $S$ is an
  orthogonal matrix:
  \begin{displaymath}
    S = S^{-1}. 
  \end{displaymath}
  In this case, condition~\eqref{eq:24} becomes symmetric with respect
  to $f$ and $f^*$:
  \begin{equation}
    \label{eq:26}
    \max\cbraces{f(x), f^*(x)} \geq \frac 12 S(x,x), \quad x\in \real{d}. 
  \end{equation}
  For every $v\in \real{d}$, the Moreau decomposition from
  \cite[Theorem~31.5, p.~338]{Rock:70} yields unique
  $x\in \dom{\partial f}$ and $y\in \partial f(x)$ such that
  $v=x+y$. Taking $u\in \real{d}\setminus G$ and setting $v=u + Su$ we
  obtain that
  \begin{displaymath}
    u+Su = x+y. 
  \end{displaymath}
  The orthogonality of $S$ yields that
  \begin{displaymath}
    Sx + Sy = Su + S^2 u= Su + u = x + y. 
  \end{displaymath}
  It follows that $x-Sy = Sx-y$ and
  \begin{displaymath}
    \abs{x-Sy}^2 = \abs{Sx-y}^2 = \ip{x-Sy}{Sx-y}.
  \end{displaymath}
  On the other hand, from~\eqref{eq:26}, the orthogonality of $S$, and
  the fact that $y\in \partial f(x)$ we deduce that
  \begin{align*}
    \frac12 \ip{x-Sy}{Sx-y} &= \frac12 \cbraces{S(x,x)+ S(y,y) -
                              \ip{x}{y} - \ip{Sx}{Sy}} \\
                            &= \frac12 \cbraces{S(x,x)+ S(y,y)} -
                              \ip{x}{y} \\
                            & \leq f(x) + f^*(y) - \ip{x}{y} = 0. 
  \end{align*}
  It follows that $Sx=y\in \partial f(x)$ and then that $x \in
  G$. Since $u + Su=x+y = x+Sx$, we have that $x-u = Su-Sx$. It
  follows that
  \begin{displaymath}
    S(x-u,x-u) = \ip{x-u}{Sx-Su} = -\abs{x-u}^2 < 0,  
  \end{displaymath}
  proving the maximality of the $S$-monotone set $G$.

  In the general case, we decompose the symmetric matrix $S$ as
  \begin{displaymath}
    S = V^T U V, 
  \end{displaymath}
  where $V$ is a $d\times d$-matrix of full rank and $U$ is symmetric
  and orthogonal. The law of inertia for quadratic forms,
  \cite[Theorem~1, p.~297]{Gant-1:98}, states that
  $U \in \msym{d}{m}$. We can actually take $U$ to be of the canonical
  form of Example~\ref{ex:4}. It is easy to see that
  \begin{displaymath}
    S(x,x)  = U(Vx,Vx), \quad x\in \real{d}, 
  \end{displaymath}
  and
  \begin{displaymath}
    G \in \mset{S} \iff F\set VG = \descr{x}{V^{-1}x\in G} \in \mset{U}. 
  \end{displaymath}
  
  We define the closed convex function
  \begin{displaymath}
    g(x) \set f(V^{-1}x), \quad x\in \real{d},      
  \end{displaymath}
  and deduce from~\eqref{eq:25} that
  \begin{displaymath}
    F = \descr{x}{f(V^{-1}x)= \frac12
      S(V^{-1}x,V^{-1}x)} = \descr{x}{g(x) = \frac12 U(x,x)}. 
  \end{displaymath}
  In view of~\eqref{eq:24} we have that
  \begin{align*}
    g(x) & \geq \frac12 S(V^{-1}x,V^{-1}x) = \frac12 U(x,x),\\
    g^*(y) &= \sup_{x\in \real{d}} \cbraces{\ip{Vx}{y} - g(Vx)} =
             \sup_{x\in \real{d}} \cbraces{\ip{x}{V^Ty} - f(x)} = f^*(V^Ty), \\
    g^*(Ux) & = f^*(V^TUx) = f^*(SV^{-1}x) \geq  \frac12 S(V^{-1}x,V^{-1}x) = 
              \frac12 U(x,x). 
  \end{align*}
  Thus, $F \in \mset{U}$ by what has been already proved.
\end{proof}

Let $G\in \mset{S}$. We recall that the Fitzpatrick function $\psi_G$
is defined as
\begin{gather*}
  \psi_G(y) = \sup_{x\in G} \cbraces{S(x,y) - \frac12 S(x,x)}, \quad y
  \in \real{d}.
\end{gather*}

\begin{Theorem}
  \label{th:12}
  Let $S \in \msym{d}{m}$ and $G\in \mset{S}$. The Fitzpatrick
  function $\psi_G$ is a closed convex function on $\real{d}$ taking
  values in $\realext$ such that
  \begin{equation}
    \label{eq:27}
    \psi_G^*(Sx) \geq \psi_G(x) \geq
    \frac12 S(x,x), \quad x\in \real{d}.    
  \end{equation}
  We have that
  \begin{equation}
    \label{eq:28}
    \begin{split}
      x\in G &\iff Sx\in \partial \psi_G(x) \iff x\in \partial
      \psi^*_G(Sx) \\
      &\iff \psi_G(x) = \frac12 S(x,x) \iff \psi_G^*(Sx) = \frac12
      S(x,x).
    \end{split}
  \end{equation}
\end{Theorem}

\begin{proof}
  Being an upper envelope of a family of linear functions, $\psi_G$ is
  a closed convex function taking values in $\realext$. We have that
  \begin{displaymath}
    S(x,x) - 2\psi_G(x) = \phi_G(x) \set \inf_{y\in G} S(x-y,x-y),
    \quad x \in \real{d}.  
  \end{displaymath}
  As $G$ is $S$-monotone, $\psi_G(x) = \frac12 S(x,x)$, $x \in G$.
  The maximal $S$-monotonicity of $G$ yields that
  $\psi_G(x) > \frac12 S(x,x)$, $x \not\in G$, as, otherwise,
  $G\cup \braces{x}$ is an $S$-monotone set.  Therefore,
  $\psi_G(x) \geq \frac12 S(x,x)$, $x \in \real{d}$, and $x\in G$ if
  and only if $\psi_G(x) = \frac12 S(x,x)$.  We now have
  \begin{align*}
    \psi_G^*(Sy)
    & = \sup_{x\in \real{d}}
      \cbraces{\ip{x}{Sy}  -
      \psi_G(x)} 
      \geq \sup_{x\in G} \cbraces{S(x,y)  -
      \frac12 S(x,x)}  \\
    & = \psi_G(y) \geq \frac12 S(y,y), \quad y \in \real{d}, 
  \end{align*}
  proving~\eqref{eq:27}.

  The second equivalence in~\eqref{eq:28} holds by the
  properties~\eqref{eq:3} of subdifferentials of convex functions.
  Theorem~\ref{th:11} yields the other equivalences, completing the
  proof.
\end{proof}

For a function $f$ taking values in $\realext$, we denote by
$\conv{f}$ its convex hull, that is the largest convex function
smaller than $f$, and by $\clconv{f}$ its \emph{closed} convex hull:
\begin{displaymath}
  \clconv{f}(x) = \liminf_{y\to x} \conv{f}(y) \set \lim_{\epsilon
    \downarrow 0} \inf_{\abs{y-x}<\epsilon} \conv{f}(y).
\end{displaymath}
We denote by $\ps{f}{\real{d}}$ the family of probability measures on
$\real{d}$ with finite support.

\begin{Theorem}
  \label{th:13}
  Let $S\in \msym{d}{m}$. For a function
  $\map{f}{\real{d}}{\realext}$, the following conditions are
  equivalent:
  \begin{enumerate}[label = {\rm (\alph{*})}, ref={\rm (\alph{*})}]
  \item \label{item:12} For every $\mu\in \ps{f}{\real{d}}$,
    \begin{displaymath}
      \mu(f) \set \int f d\mu \geq \frac{1}{2} S(\mu(x),\mu(x)).
    \end{displaymath}
  \item\label{item:13} $\clconv{f}(x) \geq \frac12 S(x,x)$,
    $x\in \real{d}$.
  \item \label{item:14} There is $G\in \mset{S}$ such that
    $f\geq \psi_G$.
  \end{enumerate}
\end{Theorem}

\begin{proof}
  \ref{item:14} $\implies$ \ref{item:12}: By Theorem~\ref{th:12},
  $\psi_G$ is convex and $\psi_G(x)\geq \frac12 S(x,x)$,
  $x \in \real{d}$. For $\mu\in \ps{f}{\real{d}}$, the convexity of
  $\psi_G$ yields that
  \begin{displaymath}
    \mu(f) \geq \mu(\psi_G) \geq \psi_G(\mu(x)) \geq
    \frac12 S(\mu(x),\mu(x)).
  \end{displaymath}
  
  \ref{item:12} $\implies$ \ref{item:13}: We recall,
  \cite[Theorem~5.6, p.~37]{Rock:70}, that
  \begin{displaymath}
    \conv f(y) = \inf\descr{\mu(f)}{\mu\in \ps{f}{\real{d}}, \;
      \mu(x)=y}, \quad y\in \real{d}.   
  \end{displaymath}
  It follows that $\conv f(x) \geq \frac12 S(x,x)$, $x\in \real{d}$.
  Finally,
  \begin{displaymath}
    \clconv{f}(x) = \liminf_{y\to x} \conv{f}(y)
    \geq \frac12 \liminf_{y\to x} 
    S(y,y) = \frac12 S(x,x). 
  \end{displaymath}

  \ref{item:13} $\implies$ \ref{item:14}: Let $\Psi$ be the family of
  closed convex functions $g=g(x)$ on $\real{d}$ such that
  $g(x) \geq \frac12 S(x,x)$. We call $h\in \Psi$ \emph{minimal} if
  there is no $g\in \Psi$ such that $g\leq h$ and $g\not=h$.

  \begin{Claim}
    \label{claim:1}
    For every $g\in \Psi$, there is a minimal $h\in \Psi$ such that
    $g\geq h$.
  \end{Claim}

  \begin{proof}
    For $h\in \Psi$, we write
    \begin{displaymath}
      \Psi(h) \set \descr{\widetilde{h}\in \Psi}{\widetilde{h}\leq h}. 
    \end{displaymath}
    Clearly, $\dom{h} \subset \dom{\widetilde{h}}$,
    $\widetilde{h}\in \Psi(h)$. By moving, if necessary, to a smaller
    element $g\in \Psi$, we can assume from the start that the domain
    of every element of $\Psi(g)$ has the same affine hull:
    \begin{displaymath}
      L \set \aff\dom{g} = \aff\dom{h}, \quad h\in \Psi(g). 
    \end{displaymath}
    
    Let $\mu$ be a Borel probability measure on $L$ having a strictly
    positive density with respect to the Lebesgue measure on $L$. For
    $h\in \Psi(g)$, we denote
    \begin{displaymath}
      \alpha(h) \set \sup_{\widetilde{h} \in\Psi(h)}
      \cbraces{\mu\cbraces{\dom{\widetilde{h}} \setminus \dom{h}}} =
      \sup_{\widetilde{h} \in\Psi(h)} 
      \cbraces{\mu(\dom{\widetilde{h}}) - \mu(\dom{h})}. 
    \end{displaymath}
    We define a sequence $(h_n)$ in $\Psi(g)$ such that $h_1 = g$,
    $h_{n+1} \in \Psi(h_n)$, and
    $ \mu\cbraces{\dom{{h_{n+1}}}\setminus \dom{h_n}} \geq
    \alpha(h_n)/2$.
    % \begin{displaymath}
    %   \mu\cbraces{\dom{{h_{n+1}}}} - \mu\cbraces{\dom{h_n}} \geq
    %   \frac12 \alpha(h_n). 
    % \end{displaymath}
    We have that
    \begin{align*}
      \alpha(h_{n+1})
      \leq \alpha(h_n) -
      \mu\cbraces{\dom{{h_{n+1}}}\setminus \dom{h_n}}
      \leq \frac12 \alpha(h_n) \leq \cbraces{\frac1{2}}^{n} \alpha(g) \to 0. 
    \end{align*}
    Being bounded below by $\frac12 S(x,x)$, the functions $(h_n)$
    converge pointwise to a convex (not necessarily closed) function
    $\widetilde{h}$. Clearly, the closure $h$ of $\widetilde{h}$
    belongs to $\Psi(g)$ and has the property: $\alpha(h) =0$.

    By the above arguments, we can assume from the start that
    $\alpha(g) = 0$ or, equivalently, that
    \begin{displaymath}
      D\set \ri \dom(g) =  \ri \dom{h}, \quad h\in \Psi(g).  
    \end{displaymath}
    We take a Borel probability measure $\nu$ on $D$ having a strictly
    positive density with respect to the Lebesgue measure on $D$ and
    such that
    \begin{displaymath}
      \int (g(x) - \frac12 S(x,x)) d\nu < \infty. 
    \end{displaymath}
    For $h\in \Psi(g)$, we denote
    \begin{displaymath}
      \beta(h) \set \sup_{\widetilde{h} \in\Psi(h)}
      \nu\cbraces{h-\widetilde h}. 
    \end{displaymath}
    We define a sequence $(h_n)$ in $\Psi(g)$ such that $h_1 = g$,
    $h_{n+1} \in \Psi(h_n)$, and
    $\nu\cbraces{h_n - h_{n+1}} \geq \beta(h_n)/2$.  We have that
    \begin{align*}
      \beta(h_{n+1})
      \leq \beta(h_n) - \nu\cbraces{h_n - h_{n+1}} 
      \leq \frac12 \beta(h_n) \leq \cbraces{\frac1{2}}^{n} \beta(g)
      \to 0.  
    \end{align*}
    The functions $(h_n)$ converge pointwise on the relatively open
    $D$ to a finite convex function $h$ on $D$. We extend $h$ by
    semi-continuity to $\cl D$ and set its values to $\infty$ outside
    of $\cl{D}$. We obtain a closed convex function.  Clearly,
    $h \in \Psi(g)$ and $\beta(h) =0$. If $\widetilde{h} \in \Psi(h)$,
    then $\widetilde{h} = h$ on $D$, the common relative interior of
    their domains. Being closed convex functions, $h$ and
    $\widetilde{h}$ are identical.  Hence, $h$ is a minimal element of
    $\Psi(g)$.
  \end{proof}

  \begin{Claim}
    \label{claim:2}
    If $h$ is a minimal element of ${\Psi}$, then
    \begin{displaymath}
      h^*(Sx) \geq \frac12 S(x,x), \quad x\in \real{d}. 
    \end{displaymath}
  \end{Claim}
  \begin{proof}
    Assume to the contrary that $h^*(Sz) < \frac12 S(z,z)$ for some
    $z\in \real{d}$ and take first $z=0$, so that
    \begin{displaymath}
      -h^*(0) = \inf_{y\in \real{d}} h(y) > 0. 
    \end{displaymath}
    Let $u(x) \set h(x) \ind{x\not=0}$ and
    $\mu \in \mathcal{P}_f(\real{d})$.  We decompose $\mu$ as
    \begin{displaymath}
      \mu = t \delta_0 + (1-t) \nu, 
    \end{displaymath}
    where $t\in [0,1]$, $\delta_0$ is the Dirac measure at $0$, and
    $\nu \in \mathcal{P}_f(\real{d})$ is such that
    $\nu(\braces{0}) = 0$. If $t=1$, then $\mu = \delta_0$ and
    \begin{displaymath}
      \mu(u) = \delta_0(u) = u(0) = 0 = \frac12 S(\delta_0(x),\delta_0(x)) =
      \frac12 S(\mu(x),\mu(x)). 
    \end{displaymath}
    If $0\leq t<1$, then by the convexity of $h$ and the fact that
    $h\geq 0$,
    \begin{align*}
      \mu(u) & = (1-t) \nu(h) \geq (1-t)h(\nu(x))  \geq
               \frac{1-t}2 \max\cbraces{0, S(\nu(x),\nu(x))} \\  
             & = \frac{1}{2(1-t)} \max\cbraces{0,
               S(\mu(x),\mu(x))} 
               \geq \frac12 S(\mu(x),\mu(x)). 
    \end{align*}
    Implication \ref{item:12} $\implies$ \ref{item:13} shows that
    $\clconv{u}\in \Psi$. As $\clconv{u} \leq u\leq h$ and
    $u(0) = 0 < h(0)$, we get a contradiction with the minimality of
    $h$.
   
    The case of general $z$ is reduced to the one above.  It is easy
    to see that the function
    \begin{displaymath}
      f(x) = h(x+z) - S(x,z) -
      \frac12 S(z,z), \quad x\in \real{d},   
    \end{displaymath}
    is a minimal element of $\Psi$.  However,
    \begin{align*}
      {f}^*(0) & = \sup_{x\in \real{d}}
                 \cbraces{-f(x)} =
                 \sup_{x\in \real{d}}
                 \cbraces{S(x,z) +
                 \frac12 S(z,z) - h(x+z)} \\ 
               & =  h^*(Sz) - \frac12 S(z,z) < 0,  
    \end{align*}
    and we again arrive to a contradiction.
  \end{proof}

  We are ready to finish the proof of the implication. By
  Claim~\ref{claim:1} we can assume from the start that $f$ is a
  minimal element of $\Psi$. Claim~\ref{claim:2} then
  yields~\eqref{eq:24}.  Theorem~\ref{th:11} shows that
  \begin{displaymath}
    G\set \descr{x\in \real{d}}{f^*(Sx) = \frac12 S(x,x)}\in
    \mset{S}.
  \end{displaymath}
  Being a closed convex function, $f$ is the conjugate of $f^*$. It
  follows that
  \begin{align*}
    f(x) & = \sup_{y\in \real{d}}
           \cbraces{\ip{x}{Sy}  - f^*(Sy)} = \sup_{y\in \real{d}}
           \cbraces{S(x,y)  - f^*(Sy)} \\
         & \geq \sup_{y\in G} \cbraces{S(x,y)  - \frac12S(y,y)} =
           \psi_G(x), \quad x\in \real{d}.   
  \end{align*}
  By Theorem~\ref{th:12}, $\psi_G \in \Psi$. The minimality of $f$
  implies that $f=\psi_G$.
\end{proof}

Finally, we describe the structure of the subdifferential of $\psi_G$
in terms of the projection function on $G$ in the $S$-space. Recall
from~\eqref{eq:7} that
\begin{align*}
  P_G(y) & \set \argmin_{x\in G} S(x-y,x-y) = 
           \argmax_{x\in G} \cbraces{S(x,y) - \frac12 S(x,x)}  \\
         & = \descr{x\in
           G}{\psi_G(y) = S(x,y) - \frac12 S(x,x)}, \quad
           y\in \real{d}.  
\end{align*}
The projection $P_G(y)$ is a closed (possibly empty) subset of $G$.

For a convex closed set $C \subset \real{d}$ and a point $x\in C$, we
denote by $N_C(x)$ the normal cone to $C$ at $x$:
\begin{displaymath}
  N_C(x) \set \descr{s \in \real{d}}{\ip{s}{y-x} \leq 0 \text{ for all } y\in C}. 
\end{displaymath}
We recall that $N_C(x) = 0$ if and only if $x$ belongs to the interior
of $C$. Otherwise, $N_C(x)$ is an unbounded closed convex cone.

For $A\subset \real{d}$, we write
\begin{displaymath}
  SA \set \descr{Sx}{x\in A}. 
\end{displaymath}

\begin{Theorem}
  \label{th:14}
  Let $S \in \msym{d}{m}$ and $G\in \mset{S}$. The multi-functions
  $P_G$ and $\partial \psi_G$ have the same domain $D$ and are related
  as
  \begin{equation}
    \label{eq:29}
    P_G(y) = \descr{x\in G}{Sx \in \partial
      \psi_G(y)} = \descr{x\in G}{y \in \partial
      \psi^*_G(Sx)}.
  \end{equation}
  With $N_{\cl{D}}$ being the normal cone to the closure of $D$, we
  have that
  \begin{equation}
    \label{eq:30}
    \partial \psi_G(y) =\clconv\cbraces{SP_G(y)} + N_{\cl{D}}(y),
    \quad y\in D. 
  \end{equation}
  In particular, if $D$ has a nonempty interior, then
  \begin{displaymath}
    \partial \psi_G(y) = \conv\cbraces{SP_G(y)}, \quad
    y\in \interior{D} = \interior{\dom{\psi_G}}.
  \end{displaymath}
\end{Theorem}

\begin{proof}
  Let $x, y \in \real{d}$.  By the properties~\eqref{eq:3} of
  subdifferentials of closed convex functions,
  \begin{displaymath}
    Sx \in \partial \psi_G(y) \iff y\in \partial
    \psi_G^*(Sx) \iff \psi_G(y) + \psi^*_G(Sx) = S(x,y).
  \end{displaymath}
  Theorem~\ref{th:12} states that $x\in G$ if and only if
  $\psi^*_G(Sx) = \frac12 S(x,x)$. Formulas~\eqref{eq:29} for $P_G(y)$
  readily follow.
  
  We denote $D\set \dom {\partial \psi _G}$. From~\eqref{eq:29} we
  deduce that $\dom{P_G} \subset D$ and
  $SP_G(y) \subset \partial \psi_G(y)$, $y\in D$.  Since
  \begin{displaymath}
    \partial \psi_G(y) =\partial \psi_G(y) + N_{\cl{D}}(y), \quad
    y\in D, 
  \end{displaymath}
  we obtain the inclusion:
  \begin{equation}
    \label{eq:31}
    \clconv\cbraces{SP_G(y)} +
    N_{\cl{D}}(y) \subset \partial \psi_G(y),
    \quad y\in \dom{P_G}. 
  \end{equation}

  To prove the opposite inclusion, we assume for a moment that
  \begin{displaymath}
    \interior{D} \not = \emptyset. 
  \end{displaymath}
  We recall that $\interior{D}=\interior{\dom{\psi _G}}$. We also
  recall that for $y\in \interior{D}$, the classical gradient
  $\nabla \psi_G(y)$ exists if and only if $\partial \psi _G (y)$ is a
  singleton. In this case,
  $\partial \psi_G(y) = \braces{\nabla \psi_G(y)}$. We denote
  \begin{displaymath}
    \dom{\nabla \psi_G} \set \descr{y\in \interior{D}}{\nabla
      \psi_G(y) \text{ exists}}. 
  \end{displaymath}
  
  We claim that
  \begin{equation}
    \label{eq:32}
    S^{-1} \nabla \psi _G(y)\in G, \quad y\in \dom{\nabla \psi _G}.
  \end{equation}
  If the claim fails for some $y\in \dom{\nabla \psi _G}$, then there
  is $\epsilon>0$ such that
  \begin{displaymath}
    G \subset \descr{x\in \real{d}}{\abs{x-x_0}\geq \epsilon}, 
  \end{displaymath}
  where $x_0 \set S^{-1} \nabla \psi _G (y)$. As
  \begin{align*}
    \psi_G(y) & = \sup_{x\in \real{d}}\cbraces{\ip{x}{y} -
                \psi_G^*(x)} = \sup_{x\in \real{d}}\cbraces{S(x,y) -
                \psi_G^*(Sx)} \\
              & = \ip{\nabla \psi_G(y)}{y} - \psi_G^*(\nabla
                \psi_G(y)) = S(x_0,y) - \psi_G^*(Sx_0), 
  \end{align*}
  the function $f(x) \set \psi ^*_G(Sx) -S(x,y)$ attains strict
  minimum at $x_0$. The convexity of $f$ yields that
  \begin{displaymath}
    f(x_0) < \inf_{\abs{x-x_0} \geq \epsilon} f(x),
  \end{displaymath}
  and we get the contradiction:
  \begin{displaymath}
    \psi _G (y) > 
    \sup_{\abs{x-x_0}\geq \epsilon} \cbraces{S(x,y)  -
      \psi ^*_G(Sx) }
    \geq 
    \sup_{x\in G} \cbraces{S(x,y)  -
      \frac12 S(x,x)} =\psi _G(y).
  \end{displaymath}
  % The claim has been proved.
 
  Now, for $y \in {D}$ we denote by $L(y)$ the set of cluster points
  of $\nabla \psi_G$:
  \begin{displaymath}
    L(y)\set\descr{z= \lim \nabla \psi _G(y_n)}{\dom{\nabla \psi
        _G}\ni y_n\rightarrow y}. 
  \end{displaymath}
  From~\eqref{eq:29}, \eqref{eq:32}, and the continuity of
  subdifferentials we obtain that
  \begin{displaymath}
    L(y)\subset SP_G(y), \quad
    y\in {D}. 
  \end{displaymath}
  According to \cite[Theorem 25.6, p.~246]{Rock:70},
  \begin{displaymath}
    L(y) \not = \emptyset \text{ and } 
    \partial \psi_G(y) = \clconv{L}(y) + N_{\cl{D}}(y), \quad
    y\in D. 
  \end{displaymath}
  It follows that
  \begin{displaymath}
    \partial \psi_G(y) \subset \clconv\cbraces{SP_G(y)} + N_{\cl{D}}(y), \quad
    y\in D.
  \end{displaymath}
  Accounting for~\eqref{eq:31}, we obtain~\eqref{eq:30}. In
  particular, $D = \dom{P_G}$.

  To verify the final relation of the theorem, we fix
  $y\in \interior{D}$. Then the subdifferential $\partial \psi_G(y)$
  is bounded together with its subset $SP_G(y)$.  As $SP_G(y)$ is
  closed, we have that
  \begin{displaymath}
    \cl\conv\cbraces{SP_G(y)} = \conv\cbraces{SP_G(y)}.
  \end{displaymath}
  Recalling that $N_{\cl{D}}(y) = 0$ for $y\in \interior{D}$, we
  obtain the result.

  We have proved~\eqref{eq:30} under the assumption that $D$ has a
  nonempty interior. The general case is reduced to this setting by
  the following arguments.  For $x_0\in \real{d}$, we observe that
  $\dom \psi _G=x_0+\dom \psi_{G-x_0}$ and
  \begin{displaymath}
    P_G(y) = x_0+P_{G-x_0}(y-x_0), \;
    \partial \psi _G(y)  =Sx_0+\partial  \psi_{G-x_0}(y-x_0), \quad
    y\in \real{d}. 
  \end{displaymath}
  Therefore, we can assume that
  $0\in G\subset \dom \psi _G\subset \aff \dom \psi _G$.  Under this
  assumption, $F\set \aff\dom \psi _G$ is a linear space. We write
  $F = E + E'$, where
  \begin{align*}
    E' & \set \descr{x'\in F}{S(x',y)=0, \; y\in F}, \\
    E & \set \descr{x\in F}{\ip{x}{x'} = 0, \; x'\in E'},   
  \end{align*}
  and denote by $S_E$ the restriction of the bilinear form $S$ on
  $E$:
  \begin{displaymath}
    S_E(x,y) \set S(x,y), \quad x,y\in E. 
  \end{displaymath}
  The Euclidean projection $G_E$ of $G$ on $E$ is a maximal
  $S_E$-monotone set and the restriction of the Fitzpatrick function
  $\psi _G$ to $E$ is the Fitzpatrick function $\psi_{G_E}$ in the
  $S_E$-space. By construction,
  $\interior{\dom{\psi_{G_E}}} \not = \emptyset$.  Therefore, the
  relation~\eqref{eq:30} holds for $\psi_{G_E}$ in the
  $S_E$-space. Being expressed in the original coordinates, this
  relation takes exactly the form \eqref{eq:30}.
\end{proof}

\section*{Acknowledgments}

We thank Dejan Slep\v{c}ev for pointing out the
reference~\cite{HasStu:89} and the related line of research. We also
thank an anonymous referee for a careful reading and thoughtful
suggestions.

\bibliographystyle{plainnat} \bibliography{../bib/finance}
\end{document}